\documentclass[12pt,reqno]{amsart}

\usepackage[margin=1in]{geometry}
\usepackage[fleqn,tbtags]{mathtools}
\usepackage[shortlabels]{enumitem}
\usepackage{graphicx}
\usepackage{amssymb}
\usepackage{amsmath}
\usepackage{amsthm}
\IfFileExists{mlmodern.sty}{\usepackage{mlmodern}}{\usepackage{lmodern}}
\usepackage{eucal}
\usepackage{microtype}
\usepackage[nodisplayskipstretch]{setspace}
\usepackage{tikz}
\usetikzlibrary{arrows.meta,positioning,calc}
\usepackage{xcolor}
\usepackage[colorlinks=true,linkcolor=red,citecolor=blue,urlcolor=red,
hypertexnames=false]{hyperref}
\usepackage[nameinlink, capitalize, noabbrev]{cleveref}
\newtheorem{theorem}{Theorem}[section]
\newtheorem{proposition}[theorem]{Proposition}
\newtheorem{lemma}[theorem]{Lemma}
\newtheorem{conjecture}[theorem]{Conjecture}
\newtheorem{corollary}[theorem]{Corollary}
\theoremstyle{definition}
\newtheorem{definition}[theorem]{Definition}

\theoremstyle{remark}
\newtheorem{remark}[theorem]{Remark}

\newcommand{\PP}{\mathbb P}
\newcommand{\C}{\mathbb C}

\newcommand{\rank}{\operatorname{rank}}
\newcommand{\Span}{\mathrm{Span}}
\newcommand{\Sym}{\mathrm{Sym}}

\newcommand{\Ann}{\operatorname{Ann}}
\newcommand{\m}{\mathfrak m}

\title{A Proof of Ebenfelt's weak SOS Conjecture
}
\author[Z. Wang]{Zhiwei Wang}
\address{Zhiwei Wang: Laboratory of Mathematics and Complex Systems (Ministry of Education)\\ School of Mathematical Sciences\\ Beijing Normal University\\ Beijing 100875\\ P. R. China}
\email{zhiwei@bnu.edu.cn}

\author[C. Yue]{Chenlong Yue}
\address{Chenlong Yue: School of Mathematical Sciences\\ Beijing Normal University\\ Beijing 100875\\ P. R. China}
\email{aystl271828@163.com}

\author[X. Zhou]{Xiangyu Zhou}
\address{Xiangyu Zhou: Institute of Mathematics\\Academy of Mathematics and Systems Science\\and Hua Loo-Keng Key
	Laboratory of Mathematics\\Chinese Academy of
	Sciences\\Beijing\\100190\\P. R. China}
\address{School of
	Mathematical Sciences, University of Chinese Academy of Sciences,
	Beijing 100049, P. R. China}
\email{xyzhou@math.ac.cn}

\date{}

\begin{document}
	\begin{abstract}
		We prove Ebenfelt's SOS conjecture, which implies the Huang--Ji--Yin gap conjecture.
	\end{abstract}
	\maketitle
	\tableofcontents
	
	\section{Introduction}\label{sec:introduction}
	Let $n$ be an integar , $z=(z_1,\cdots, z_n)$ be the complex coordinates of $\mathbb C^n$ and   $A(z,\bar z)\in \mathbb{C}[z,\bar z]$ be a (real valued) Hermitian  polynomial. Then there exists a Hermitian matrix $H$ such that $A(z,\bar{z})=\mathfrak{Z}H\mathfrak{Z}^*$ , here the row vector 
	$$\mathfrak{Z}=(1,z_1,\cdots,z_n,z_1^2,\cdots,z_1z_n,\cdots,z_n^d)$$
	is a basis of the polynomials in $z$ of degree at most $d$ in left lexicographic order 
	, where $\mathfrak{Z}^*$ is the conjugate transpose of $\mathfrak{Z}$.
	When we refer to the rank of a Hermitian polynomial, we actually mean the corresponding matrix. Conversely, when talking about the properties of the matrix, we can also associate them with the Hermitian polynomial.  
	
	We write $A\in \mbox{SOS}_n$  if $A(z,\bar z)=\sum_{\mu=1}^\rho |f_\mu|^2$ is a sum of norm squares of holomorphic polynomials $f_\mu$. 	A basic fact is that a Hermitian polynomial $A(z,\bar z) \in \mbox{SOS}_n$ if and only if its associated Hermitian matrix $H$ is positive semidefinite , see for example \cite{Dan02}.  If the representation is minimal, then the $f_\mu$ are linearly independent, and $\rho$ is the rank of the Hermitian coefficient matrix of $A$. 
	
	Let  $\|z\|$ be the usual Euclidean norm in $\C ^n$. As a  Hermitian form of  Hilbert's 17th problem \cite{Hi00,Hi02,Ar27}, 
	Quillen \cite{Qu68} proved that, if a bihomogeneous Hermitian polynomial is strictly positive away from the origin, then there exists an integer $N$ such that $A(z,\bar{z})\|z\|^{2N}$ is a sum of norm squares. We call $A(z,\bar{z})\|z\|^{2N}$ the $N$-th prolongation of $A(z,\bar {z})$. Reccently, there are many interesting results on Hermitian analogue of Hilbert's 17th problem, see, for example \cite{CD96,Dan02,Dan05,Dan06,DLP07,V08,DL09,Dan11,BGS22} and the references therein.
	More recently, a related problem regarding the rank of the first prolongation of $A$ has been
	posed by Ebenfelt \cite{Ebe17}.  
	
	For $n\geq 2$, let $\kappa=\kappa(n)$ be the unique integer such that $\binom{\kappa}{2}<n\leq \binom{\kappa+1}{2}$, and put $\Lambda_n=n\kappa-\binom{\kappa}{2}-1$. 
	
	\begin{conjecture}[SOS conjecture, {\cite[Conjecture 1.2]{Ebe17}}]\label{conj: sos}
		For $n\geq 2$, let $A(z,\bar z)$ be a Hermitian polynomial such that $A(z,\bar z)\|z\|^2\in \mbox{SOS}_n$ , and  $\rho$ donotes the rank of $A(z,\bar z)\|z\|^2$. Then either 
		\begin{equation}\label{eq1-1}
			\rho\geq \Lambda_n,
		\end{equation}
		or,   for some $k\in\{0,\ldots,\kappa-1\}$,
		\begin{equation}\label{eq1-2}
			kn-\frac{k(k-1)}{2}\leq \rho\leq kn.
		\end{equation}
	\end{conjecture}
	
	The SOS conjecture is motivated by Huang’s lemma \cite{Hua99}  and the Huang–Ji–Yin (HJY)  Gap
	Conjecture \cite{HJY} on the  proper holomorphic  maps between the complex unit balls. 	For each $1 \leq k \leq\kappa-1$, define the interval
	$$
	I_k := \left[ kn + 1,\, (k+1)n - \frac{k(k+1)}{2} - 1 \right].
	$$
	
	\begin{conjecture}[Huang-Ji-Yin Gap conjecture \cite{HJY}]\label{conj: HJY}
		Let $n \geq 3, N \geq n$ and let $I_k$ ($1 \leq k \leq \kappa-1$) be as above.
		Then any proper holomorphic  map $F: \mathbb{B}^n \to \mathbb{B}^N$ which are three times  differentiable up to the boundary,   is equivalent to a map of the form $(G, 0')$ if and only if $N \in I_k$ for some $1 \leq k \leq \kappa-1$.
	\end{conjecture}
	HJY Gap conjecture is one of the major open problem in CR geometry. The "only if "  part follows from \cite[Theorem 2.8]{HJY}, also the HJY Gap conjecture holds for $k=1,2,3$ \cite{Hua99,HJX06,HJY14}. More recent advances are refered to \cite{GN24,YY26}  and the references therein. Ebenfelt \cite{Ebe17} showed that  that the HJY Gap Conjecture is a
	consequence of the SOS conjecture through a CR version of the Gauss equation.

	A brilliant lemma by Huang \cite{Hua99} shows that $\rho=0$ or $ \rho \geq n$, so the SOS conjecture holds for
	$n=2$ by Huang’s lemma. When $A(z,\bar z)$  itself is a sum of norm squares, the homogeneous case  was proved by Grundmeier and Halfpap   \cite{GH15} and was subsequently completed by Ebenfelt \cite{Ebe15}.   Gao and   Ng \cite{GN23} made a breakthrough, using geometric methods to demonstrate the existence of gaps under more general conditions. Subsequently, these methods were further developed to make progress to the HJY Gap Conjecture \cite{GN24}. 	Based on the result on the case that  $A(z,\bar z)$ is a sum of norm squares, Ebenfelt \cite{Ebe17} pointed out that an optimistic view of the situation in the conjecture would be to hope that the ``gaps'' in linear ranks predicted in \eqref{eq1-2} can only occur when \( A(z,\bar{z}) \) is  a sum of norm squares. Furthermore, if \( A(z,\bar{z}) \) is not  a sum of norm squares but \( A(z,\bar{z})\|z\|^2 \) is  a sum of norm squares, then \eqref{eq1-1} always holds. This is named the weak (alternative) SOS conjecture. 
	
	\begin{conjecture}[Weak (alternative) SOS conjecture, {\cite[Conjecture 1.5]{Ebe17}}]\label{conj: weak sos}
		For $n\geq 3$, if $A(z,\bar z) \notin \mbox{SOS}_n $ , but $A(z,\bar z)\|z\|^2 \in \mbox{SOS}_n$, then the rank of $A(z,\bar z)\|z\|^2$ is greater than or equal to $\Lambda_n$.
	\end{conjecture}
	
	As pointed out by Ebenfelt \cite{Ebe17}, one of the main difficulties in \cref{conj: weak sos} lies in characterizing when \(A(z,\bar z)\|z\|^2\) is in fact a sum of norm squares.
	
	There are several partial results towards the weak SOS conjecture.  In the case of  $n=3$ and the Hermitian metrix of $A(z,\bar z)$ is diagonal, it is proved by Brooks--Grundmeier \cite{BG21}. In the same diagonal case  and for  $4\leq n\leq 6$, it is proved by the authors in \cite{WYZ25} as well as a universal bound for $n\geq 7$.  In \cite{WYZ26}, the authors proposed a Newton-Okounkov body approach to study the weak SOS conjecture.

	The  first main result of the present paper is  the proof of  Ebenfelt's Weak SOS Conjecture in full generality. 
	
	\begin{theorem}
		\label{thm:main}
		Ebenfelt's weak SOS conjecture  holds. 
		
	\end{theorem}
	
	The proof of \cref{thm:main} is completed in Section 7-8. 
	Combining the work of Huang \cite{Hua99}, Grundmeier-Halfpap  \cite{GH15} and Ebenfelt \cite{Ebe15} with
	\cref{thm:main} gives the complete Ebenfelt's SOS conjecture.

	\begin{theorem}
		Ebenfelt's SOS conjecture holds.
	\end{theorem}

	As already mentioned above, via a CR version of the Gauss equation, Ebenfelt's SOS conjecture implies the the HJY Gap conjecture. This is the second main result of the present paper.

	\begin{theorem}\label{cor:HJY}
		Huang-Ji-Yin's Gap conjecture holds.
	\end{theorem}

	\noindent\textbf{Strategy.} As pointed out by Ebenfelt, the key point to study the weak SOS conjecture is to find the  additional structure hidden in the SOS identity.

	To see this structure, fix a minimal representation
	\[
	A(z,\bar z)\|z\|^2=\sum_{\mu=1}^{\rho}|h_\mu(z)|^2,
	\qquad H=\Span\{h_1,\ldots,h_\rho\}.
	\]
	One of the key observations comes from polarization. Polarization of the SOS identity gives
	\[
	\sum_{\mu=1}^{\rho}h_\mu(z)h_\mu^*(\xi)
	=L(z,\xi)A(z,\xi),
	\qquad L(z,\xi)=\sum_{j=1}^n z_j\xi_j,
	\]
	where $h^*_\mu(\xi):=\overline{h_\mu(\bar\xi)}$. From the above identity, we observe that  \(L(\cdot,\xi)A(\cdot,\xi)\in H\) for every \(\xi\): the space \(H\) contains a family of elements divisible by a moving linear form, say $L(\cdot, \xi)$ for $\xi\in \mathbb C^n$.  
	This is done in \cref{sect: macau}. In this section, we also identified two necessary conditions from weak SOS hypothesis. The first is that \(A(z,\bar z)\) has holomorphic degree at least two. The second is that the rank of \(A(z,\bar z)\) is \(\geq \kappa+1\) (see \cref{prop:triangular-rank}).
	
	This observation motivates us to exploit this moving divisibility by restricting \(H\) successively along a general flag \(\C^n=L_n\supset L_{n-1}\supset\cdots\supset L_2\), where $L_p$($ 2\leq p\leq n$) are $p$-dimensional subspaces of $\mathbb C^n$. Define \(\rho_p=\dim(H|_{L_p})\) and \(\delta_p=\rho_{p+1}-\rho_p\). We are aiming to estimate \(\rho=\rho_n\). Note that \(\rho_n=\sum_{i=p}^{n-1}\delta_i+\rho_p\),  it suffices to estimate \(\delta_i\) and \(\rho_p\) for \(p\) small. For this, we develop the theory of admissible flags and analyse restriction dynamics along an admissible flag, where the assumption \(A(z,\bar z)\not\in \mbox{SOS}_n\) is repeatedly used to exclude many exotic cases,  this is done in \cref{sect: big cell}-\cref{sect: divi dim}. In particular, in \cref{lem:triangular-staircase}, we find that when \(n\geq 2\kappa\) (this is only possible for \(\kappa\geq 3\)), we can get a better estimate. Then we divide the estimate of \(\rho\) into two cases: \(n\geq 2\kappa\) (i.e. \(n=6\) and \(n\geq 8\)), and \(n<2\kappa\) (i.e. \(n=3,4,5,7\)).
	
	The other key  ingredient is orthogonality. For any \(p\)-plane \(L\subset\C^n\) and its orthogonal complement $L^\perp$, the polarized identity reads that, for any $z\in L, w\in L^\perp$, $\sum_{\mu=1}^\rho h_\mu	(z)\overline{h_\mu(w)}=L(z,\bar w)A(z,\bar w)=0$. This means that the subspaces $\Span\{\mathbf h(z):=(h_1(z),\cdots, h_\rho(z))\in \mathbb C^\rho:z\in L\}$ and $\Span\{\mathbf h(z):=(h_1(z),\cdots, h_\rho(z))\in \mathbb C^\rho:z\in L^\perp\}$ are orthogonal, and their dimensions equals to $\dim (H|_L)$ and  $\dim (H|_{L^\perp})$, respectively.  By choosing $L$ and $L^\perp$ general, we obtain that $\rho\ge \dim (H|_L)+\dim (H|_{L^\perp})=\rho_p+\rho_{n-p}$. In the case $n\geq 2\kappa$, estimating $\rho_\kappa$ and $\rho_{n-\kappa}$ finished the proof, this is done in \cref{sect: big dim}. The inequality is also frequently used in the proof of  the case of  the low dimensionnal cases $n=3,4,5,7$, this is done in \cref{sect: small dim}.

	\vspace{1.5em}
	\noindent\textbf{Acknowledgements.}
	The first and second authors wish to thank Professors Xiaojun Huang, Wanke Yin, Sui-Chung Ng, and Yun Gao for helpful discussions, and in particular Professor Xiaojun Huang for his invaluable encouragement, care, and support throughout the preparation of this paper.  This research is supported by the National Key R\&D Program of China (Grant Nos. 2021YFA1002600 and 2021YFA1003100). Z. Wang and X. Zhou are partially supported by the National Natural Science Foundation of China (Grant Nos. 12571085 and 12288201), respectively. Z. Wang is also supported by the Fundamental Research Funds for the Central Universities.

	\section{Macaulay Estimates and Polarized SOS Identities}\label{sect: macau}
	In this section, we collect some Macaulay-type estimates and derive the auxiliary estimates needed later from the hypotheses of the weak SOS conjecture.

	Throughout this paper, we work on the number field $\mathbb C$. 	Let $S=\C[z_1,\ldots,z_n], E=S_1$ be the $n$ dimensional space of linear forms, and let $S_d$ denote the  homogeneous part of degree $d$ ,  thus $S=\Sym(E)=\bigoplus_{d\ge 0} S_d$. For a finite dimensional polynomial space $U\subset S$, define $E \cdot U =\Span\{\ell u:\ell\in E,\ u\in U\}$ which is denoted by $EU$ for simplicity. Let $K_1(U)$ be the kernel of the multiplication map $E\otimes U\to E \cdot U$, and set
	\[
	k_U:=\dim K_1(U)=n\dim U-\dim (E U).
	\]
	
	We begin by recalling the Macaulay representation of integers. For positive integers $N$ and $d$, there is a unique sequence of integers $k_d>k_{d-1}>\cdots>k_\delta\geq\delta\geq1$ such that
	\[
	N=\binom{k_d}{d}+\binom{k_{d-1}}{d-1}+\cdots+\binom{k_\delta}{\delta}.
	\]
	This is the Macaulay representation of $N$ of order $d$. Define
	\[
	N^{\langle d\rangle}=\binom{k_d+1}{d+1}+\binom{k_{d-1}+1}{d}+\cdots+\binom{k_\delta+1}{\delta+1},
	\]
	and set $0^{\langle d\rangle}=0$ and $\binom{a}{b}=0$ if $a < b$.

	We have the following Macaulay estimate.
	
	\begin{theorem}[{\cite{Macaulay1927}}]\label{thm: mac est}
		For every $d \geq 1 $, let $U$ be a subspace of $S_d$. Then
		\[
		\mathrm{codim}(EU)\leq\mathrm{codim}(U)^{\langle d\rangle}.
		\]
	\end{theorem}
	Here the codimensions of $U$ and $E U$ are taken relative to $S_d,S_{d+1}$, respectively.
	For $0\le r\le2n-1$, define
	\begin{equation}
		\gamma_n(r)=
		\begin{cases}
			nr-\binom r2, & 0\le r\le n,\\[3pt]
			\binom{n+1}{2}+(n-1)a-\binom a2, & r=n+a,\quad 1\le a\le n-1.
		\end{cases}
	\end{equation}
	
	\begin{lemma}\label{lem:filtered-macaulay}
		If $U\subset S$  and $\dim U=r\le2n-1$, then $\dim EU\ge\gamma_n(r)$.
	\end{lemma}
	\begin{proof}
		First,  suppose that  $U$ is homogeneous, and that $U\subset S_d$. If $d=0$, then $r\leq1$ and $\dim E \cdot U=nr=\gamma_n(r)$, so there is nothing to prove. Assume henceforth that $d\geq1$. For $r\leq n$, the desired estimate is \cite[Proposition 3 ]{GH15}, since both cases  follow from the same calculation, for completeness,  we include the proof here. Write
		\[\dim S_d=\binom{n+d-1}{d}=\sum_{i=1}^d\binom{n+i-2}{i}+1,\]
		where an empty sum is understood to be zero. For $1\leq r\leq n$ and $n<r=n+a\leq2n-1$, respectively, the Macaulay representations of the codimension are
		\begin{align*}\mathrm{codim}(U)&=\dim S_d-r=\sum_{i=2}^d\binom{n+i-2}{i}+\binom{n-r}{1},\\
			\mathrm{codim}(U)&=\dim S_d-(n+a)=\sum_{i=3}^d\binom{n+i-2}{i}+\binom{n-1}{2}+\binom{n-1-a}{1}.\end{align*}
		Applying \cref{thm: mac est} and subtracting the resulting upper bound from $\dim S_{d+1}$ gives, in the two ranges, $\dim EU\geq nr-\binom r2$ and
		\(
		\dim EU\geq\binom{n+1}{2}+(n-1)a-\binom a2,
		\)
		respectively. Thus $\dim EU\geq\gamma_n(r)$ in the homogeneous case.
		
		Now suppose that $U$ is nonhomogeneous, and fix a lexicographic monomial order, the estimate considered here is a special case of the corresponding estimate in the work of \cite{Gre15}. Apply  elimination to a basis $f_1,\dots,f_r$ of $U$ to obtain distinct leading monomials $m_i=\mathrm{in}(f_i)$, and define the initial space by
		\[\mathrm{in}(U):=\Span\{\mathrm{in}(f):f\in U\}=\Span\{m_1,\ldots,m_r\}.\]
		The initial space is independent of the choice of basis and its dimensin  equals  $\dim (U)$. Every monomial in $E \cdot \mathrm{in}(U)$ occurs as the leading monomial of an element of $EU$. Hence $E \cdot \mathrm{in}(U)\subset\mathrm{in}(EU)$ and $$\dim(E \cdot \mathrm{in}(U)) \leq \dim (\mathrm{in}(EU)) = \dim EU. $$ Let $\mathrm{in}(U)_d=\mathrm{in}(U)\cap S_d$ and write $\dim\mathrm{in}(U)_d=r_d$. Since it is spanned by monomials, the initial space decomposes by degree as $\mathrm{in}(U)=\bigoplus_{d\geq0}\mathrm{in}(U)_d$.
		Grouping the $m_i$ by total degree, the spaces $E \cdot \mathrm{in}(U)_d$ lie in distinct degrees, so
		\begin{align*}
			\dim (E \cdot \mathrm{in}(U)) = \sum_{d \geq 0} \dim \big( E \cdot \mathrm{in}(U)_d \big)
			\geq\sum_{d \geq 0}\gamma_n(r_d).
		\end{align*}
		The function $\gamma_n(r)$ is subadditive, a direct calculation verifies that
		\[\gamma_n(a+b)\leq\gamma_n(a)+\gamma_n(b),\qquad a+b\leq2n-1,\quad a,b\geq0.\]
		Therefore $\dim EU\geq\sum_d\gamma_n(r_d)\geq\gamma_n(\sum_dr_d)=\gamma_n(r)$.
	\end{proof}

	We also need an estimate without a dimension restriction. The following lemma follows readily from the initial space and degree decomposition argument.
	
	\begin{lemma}\label{lem:two-variable}
		If $0\ne U\subset\C[x,y]$ is finite dimensional, then $\dim(xU+yU)\ge\dim U+1$.
	\end{lemma}
	\begin{proof}
		Fix a monomial order $x \succ y$ and let $\operatorname{in}(U)$ be the initial
		space of $U$. As we showed above 
		$\dim(xU+yU)\ge
		\dim(x\,\operatorname{in}(U)+y\,\operatorname{in}(U))$.
		It is therefore enough to prove the assertion when $U$ is spanned by
		monomials.
		
		Write $\operatorname{in}(U)=\bigoplus_d U_d$, where $U_d$ is spanned
		by monomials of degree $d$, and put $r_d=\dim U_d$. The spaces
		$xU_d+yU_d$ lie in distinct homogeneous parts as $d$ varaies , so their dimensions add: $\dim(x\,\operatorname{in}(U)+y\,\operatorname{in}(U)) = \sum_d \dim(x\,\operatorname{in}(U)_d+y\,\operatorname{in}(U)_d)$.
		
		Fix $d$ with $r_d>0$. The degree-$d$ monomials in two variables form
		the ordered chain
		$x^d\succ x^{d-1}y\succ \ldots \succ y^d$.
		Let $M_d$ be the set of the $r_d$ monomials spanning $U_d$. Both
		$xM_d$ and $yM_d$ contain $r_d$ monomials. The largest monomial in $xM_d$ does not belong to $yM_d$: otherwise, dividing by $y$ would produce a monomial in $M_d$ larger than the largest monomial of $M_d$. Therefore
		\[
		|xM_d\cup yM_d|
		\ge |yM_d|+1=r_d+1.
		\]
		
		Summing over all degrees for which $U_d\ne0$, and denoting their
		number by $s$, gives
		$\dim(x\,\operatorname{in}(U)+y\,\operatorname{in}(U))
		\ge\dim U+s$.
		Since $U\ne0$, we have $s\ge1$, and hence
		$\dim(xU+yU)\ge\dim U+1$.
	\end{proof}
	
	Now we try to start from the conditions of the weak SOS conjecture and uncover the mathematical information implicit in them.
	Take a minimal inertia decomposition of $A \notin SOS_n$:
	\[
	A(z,\bar z)
	=
	\sum_{\alpha=1}^{P}|f_\alpha(z)|^2
	-
	\sum_{a=1}^{N}|g_a(z)|^2,
	\]
	and let $V=\Span\{f_\alpha\}$ and $G=\Span\{g_a\}$. Minimality means that the
	$f_\alpha$ and $g_a$ are jointly linearly independent. Consequently, $V\cap G=0$,
	$P+N=\rank A$, and $N\ge1$.
	
	For the first prolongation considered here, we recall a standard fact  : if $A(z,\bar z)\|z\|^2$ is a Hermitian sum of squares, then
	\begin{equation}\label{eq:standard-inclusion}
		EG\subset EV.
	\end{equation}
	Its bihomogeneous version see  \cite[Proposition~1]{BGS22} , the bihomogeneous version at arbitrary prolongation order appears in \cite[Lemma~1]{GH15} and  more general form goes back to \cite{Dan05}. We give a  proof for completeness.
	\begin{lemma}
		\label{lem:first order-prolongation}
		If
		\(A(z,\bar z)\|z\|^2\in\mathrm{SOS}_n,\)
		then $EG\subset EV.$	
	\end{lemma}
	\begin{proof}
		Choose a  finite dimension monomial subspace (i.e. spanned by monomials) \(W \subset S \) containing $EV + EG$  , and equip
		\( W\) with the Hermitian inner product for which its monomial
		basis is orthonormal. Declare the bases
		\(z_j\otimes f_\alpha\)
		of \(E\otimes V\) and
		\(z_j\otimes g_a\)
		of \(E\otimes G\) to be orthonormal, and let	
		\[
		T_V:  E \otimes V \longrightarrow W,
		\qquad
		T_G:  E \otimes G \longrightarrow W,
		\]
		In the bases just specified , let \(T_V\) and
		\(T_G\) be the matrices whose columns are the coefficient vectors of
		the polynomials \(z_jf_\alpha\) and \(z_jg_\beta\) in terms of monomial basis of $W$, respectively.
		Then
		\[
		A(z,\bar z)\|z\|^2
		=
		\mathbf m(z)^T
		\bigl(T_VT_V^*-T_GT_G^*\bigr)
		\overline{\mathbf m(z)},
		\]
		where \(\mathbf m(z)\) is the column vector of monomial basis of $W$, and $T_V^*, T_G^*$ are conjugate transpose of the matrices $T_V, T_G$, respectively. Hence the coefficient Hermitian matrix of
		\(A(z,\bar z)\|z\|^2\) is
		\[
		M=T_VT_V^*-T_GT_G^*.
		\]
		By hypothesis, \(A(z,\bar z)\|z\|^2\) is a Hermitian sum of squares, hence
		\(M\ge0\).
		Now take \(h\in(EV)^\perp\). Since \(\operatorname{Range}T_V=EV\), we have
		\(T_V^*h=0\). Therefore
		\[
		0\le \langle Mh,h\rangle
		=
		\|T_V^*h\|^2-\|T_G^*h\|^2
		=
		-\|T_G^*h\|^2.
		\]
		It follows that \(T_G^*h=0\), and hence
		\(h\in(EG)^\perp\). Thus
		\(
		(EV)^\perp\subset(EG)^\perp.
		\)
		Taking orthogonal complements gives \cref{eq:standard-inclusion}.
	\end{proof}
	
	For a holomorphic polynomial \(f\), write
	\(f^*(\xi)=\overline{f(\bar\xi)}\), and denote by \(A(z,\xi)\) the
	polarization of \(A(z,\bar z)\). Fix a minimal representation
	\(A\|z\|^2=\sum_{\mu=1}^{\rho}|h_\mu|^2\), and let
	\begin{equation}
		H=\Span\{h_1,\ldots,h_\rho\},
	\end{equation}
	so that \(\dim H=\rho\). With
	\(L(z,\xi)=\sum_{j=1}^n z_j\xi_j\), polarization gives
	\begin{equation}\label{eq:polarization}
		\sum_{\mu=1}^\rho h_\mu(z)h_\mu^*(\xi)=L(z,\xi)A(z,\xi).
	\end{equation}

	We first remove a global common factor. Let
	\(q=\gcd(h_1,\ldots,h_\rho)\), and write \(h_\mu=qh'_\mu\). The polarized
	identity becomes
	\[
	L(z,\xi)A(z,\xi)=q(z)q^*(\xi)\sum_\mu h'_\mu(z)(h'_\mu)^*(\xi).
	\]
	For \(n\ge2\), the bilinear polynomial \(L(z,w)\) is irreducible. Moreover, it
	cannot divide \(q(z)q^*(\xi)\), because neither factor involves both groups
	of variables. Euclid's lemma therefore shows that \(L(z,w)\) divides the final
	sum. Write
	\[
	\sum_\mu h'_\mu(z)(h'_\mu)^*(\xi)=L(z,\xi)A_0(z,\xi).
	\]
	The numerator and \(L(z,w)\) are Hermitian under interchange of the two
	polarized variable groups, so \(A_0\) is Hermitian. Canceling \(L\) gives
	\[
	A=|q|^2A_0,
	\qquad
	A_0\|z\|^2=\sum_{\mu=1}^{\rho}|h'_\mu|^2.
	\]
	Indeed,
	\(A(\mathord\cdot,\xi)=q(\cdot) q^*(\xi)A_0(\mathord\cdot,\xi)\). Restricting
	\(\xi\) to the dense open set \(\{q^*(\xi)\neq0\}\) does not change the
	span of the values of the polynomial family
	\(A_0(\mathord\cdot,\xi)\): any linear functional that vanishes on that
	dense set vanishes identically. It follows that the coefficient spaces of \(A\) and
	\(A_0\) satisfy
	\[
	\Span_\xi A(\mathord\cdot,\xi):=\Span\{A(\cdot,\xi): \xi\in \mathbb C^n\}
	=q\Span\{A_0(\cdot,\xi): \xi\in \mathbb C^n\}=:\Span_\xi A_0(\mathord\cdot,\xi).
	\]
	Since the coefficient rank of a Hermitian polynomial
	equals the dimension of its coefficient space, multiplication by the
	nonzero polynomial \(q\) preserves \(\rank A\). It is also injective, so
	the \(h'_\mu\) remain linearly independent and
	\[
	\rank(A_0\|z\|^2)=\rank(A\|z\|^2)=\rho,
	\qquad
	\rank A_0=\rank A.
	\]
	Furthermore, \(A_0\notin\mathrm{SOS}_n\), since otherwise
	\(A=|q|^2A_0\) would belong to \(\mathrm{SOS}_n\). Thus replacing \(A\) by
	\(A_0\) preserves every rank and hypothesis used below. We may therefore
	assume that \(H\) is primitive, meaning that its elements have no
	nonconstant common factor.
	
	Since \(A(z,\bar z)\|z\|^2\) vanishes at \(z=0\), every \(h_\mu\)
	vanishes at the origin, so \(H\subset\m=(z_1,\ldots,z_n)\). For each \(\xi\in \mathbb C^n\), the
	standard element \(L(z,\xi)A(z,\xi)\) belongs to \(H\). These elements span
	\(H\): indeed, if the vectors
	\((h_1^*(\xi),\ldots,h_\rho^*(\xi))\) failed to span \(\C^\rho\) as
	\(\xi\) varies, there would be a nonzero linear relation among the
	\(h_\mu\).

	Let
	\begin{equation}
		V_A=\Span_\xi A(\cdot,\xi).
	\end{equation}
	
	Then \(\dim V_A=\rank A\), and \(V_A\) is primitive: a nonconstant common
	factor of \(V_A\) would divide every standard element and hence every
	element of \(H\). Denote by \(A_{d,d}\) the component of \(A\) of bidegree
	\((d,d)\), and by \(h_{\mu,d}\) the degree \(d\) homogeneous component of
	\(h_\mu\). Taking the component of bidegree \((d+1,d+1)\) in
	\eqref{eq:polarization} gives
	\begin{equation}\label{eq:homogeneous-polarized-identity}
		L(z,\xi)A_{d,d}(z,\xi)=\sum_{\mu=1}^\rho h_{\mu,d+1}(z)h_{\mu,d+1}^*(\xi).
	\end{equation}
	\begin{remark}\label{remark:degree}
		A pointwise nonnegative Hermitian polynomial of holomorphic degree at most one has a positive semidefinite coefficient matrix and is itself a sum of squared norms. The sum of squares hypothesis implies that $A(z,\bar z)$ is pointwise nonnegative. Thus, in the weak SOS setting, $A$ has holomorphic degree at least two. In other words, if $A\not\in \mbox{SOS}_n$, then $V_A\not\subset S_1$.
	\end{remark}
	For $n\ge2$, let $\kappa=\kappa(n)$ be the unique integer satisfying $\binom \kappa2<n\le\binom{\kappa+1}2$.  The homogeneous case of the following propositon with negative index $N=1$ is \cite[Proposition~4]{GH15}.
	
	\begin{proposition}\label{prop:triangular-rank}
		Under the hypotheses of \cref{thm:main}, one has $\rank A\ge \kappa+1$.
	\end{proposition}
	
	\begin{proof}
		Let $W=V\oplus G$ and $r_A=\dim W=\rank A$. By \cref{lem:first order-prolongation}, we have $EW=EV$, and therefore
		\[
		k_W=n(P+N)-\dim EW=nP-\dim EV + nN=k_V+nN\ge n.
		\]
		Suppose, to the contrary, that $r_A \le \kappa$. Since $\binom \kappa2<n$, we have $\kappa\le n$, so the first branch of \cref{lem:filtered-macaulay}  gives $\dim EW\ge nr_A-\binom{r_A}{2}$. Hence $k_W\le\binom{r_A}{2}\le\binom \kappa2<n$, a contradiction.
	\end{proof}

	\section{Big Cells and Genericity }\label{sect: big cell}
	
	As mentioned in \cref{sect: macau}, the space \(H\) is spaned by the
	standard elements \(L(z,\xi)A(z,\xi)\)  for all \(\xi\in \mathbb C^n\). This motivates us to exploit this moving divisibility by restricting \(H\) successively along a general flag \(\C^n=L_n\supset L_{n-1}\supset\cdots\supset L_2\), where $L_p, 2\leq p\leq n$, are $p$-dimensional subspaces of $\mathbb C^n$. 
	In this section, we   collect  some basic properties of flag manifolds and prepare the mathematical language and tools for studying these sequences of restrictions.
	We refer to \cite{Bri05} for the relevant background. 
	
	Let $V$ be an $n$ dimensional complex vector space, with $n\ge3$.  We
	write
	$$
	\mathcal F(V)=\operatorname{Fl}(n-1,\ldots,3,2;V)
	$$
	for the manifold of flags
	$F=(L_{n-1},\ldots,L_2)$ satisfying
	$V=L_n\supset L_{n-1}\supset\cdots\supset L_2$ and
	$\dim L_p=p$.  Then  $\mathcal F(V)$ is a  closed
	subvariety of
	$\prod_{p=2}^{n-1}\operatorname{Gr}(p,V)$ and 
	an irreducible projective algebraic variety.
	
	We say that a function depend regularly on  the parameters if, locally on the parameter space, it is a quotient of two polynomial functions whose denominator does not vanish.
	
	\medskip
	Fix a basis $e_1,\ldots,e_n$ of $V$, with dual basis
	$z_1,\ldots,z_n$ of $V^*$, and put
	$L_p^0=\Span \{e_1,\ldots,e_p\}$.
	For $2\le p\le n-1$, let
	$$
	\operatorname{pr}_p:V\longrightarrow L_p^0
	$$
	denote the   projection  along $\Span\{e_{p+1},\ldots,e_n\}$.
	
	Consider the natural regular projection
	$$
	\pi_p:\mathcal F(V)\longrightarrow\operatorname{Gr}(p,V),
	\qquad
	F\longmapsto L_p.
	$$
	Let $\mathcal C_p\subset\operatorname{Gr}(p,V)$ be the
	Grassmannian big cell, namely the set of $p$-planes $L$ for which
	$\operatorname{pr}_p|_{L}$ is an isomorphism, this  is equivalant to  the determinant of the first $p$ rows of
	homegeneous coordinates of $L$ in $Gr(p,V) $ is nonzero. Thus, $\mathcal{C}_p$ is a nonempty Zariski open set in $\operatorname{Gr}(p,V)$.
	
	We define the \textbf{ big cell} $\mathcal C\subset\mathcal F(V)$ by
	\[ \mathcal{C}= \bigcap_{p=2}^{n-1} \pi_p^{-1}(\mathcal{C}_p).\]
	
	Since $\pi_p$ is a regular map, the inverse image
	$\pi_p^{-1}(\mathcal C_p)$ is a nonempty Zariski open subset of
	$\mathcal F(V)$. Taking the finite intersection over all relevant
	$p$ yields a Zariski open subset $\mathcal C\subset\mathcal F(V)$.
	Moreover, $\mathcal C$ contains the standard flag
	$L_2^0\subset\cdots\subset L_{n-1}^0$, and hence is nonempty.
	Finally, since $\mathcal F(V)$ is irreducible, every nonempty Zariski
	open subset is dense, so $\mathcal C$ is dense in $\mathcal F(V)$.
	
	The same condition gives explicit affine coordinates.  For
	$F\in\mathcal C$, there is a unique basis
	$v_1(F),\ldots,v_n(F)$ of $V$ such that
	$L_p=\Span\{v_1(F),\ldots,v_p(F)\}$ for
	$2\le p\le n-1$ and whose coefficient matrix
	$B(F)$ relative to $e_1,\ldots,e_n$ is lower unitriangular  , the $j$th column of $B(F)$ is the coordinate
	vector of $v_j(F)$ and $b_{21}=0$:
	\begin{equation}\label{eq:big cell}
		B(F)=
		\begin{pmatrix}
			1&0&0&\cdots&0\\
			0&1&0&\cdots&0\\
			b_{31}&b_{32}&1&\cdots&0\\
			\vdots&\vdots&\ddots&\ddots&\vdots\\
			b_{n1}&b_{n2}&\cdots&b_{n,n-1}&1
		\end{pmatrix}.
	\end{equation}
	
	Conversely, every matrix of the  form displayed in (\ref{eq:big cell}) determines a flag in
	$\mathcal C$ by taking successive column spans.  Hence the entries
	$b_{ij}$, with $3\le i\le n$ and $1\le j<i$, are affine coordinates
	on $\mathcal C$.  In particular,
	$\mathcal C\simeq  {\mathbb{C}}^{n(n-1)/2-1}$.
	
	Let
	$\ell_1(F),\ldots,\ell_n(F)$ be the dual coframe.  Then
	$\ell_i(F)(v_j(F))=\delta_{ij}$ for all $i,j$. Write
	\[
	\begin{pmatrix}
		\ell_1(F)\\
		\vdots\\
		\ell_n(F)
	\end{pmatrix}
	=
	B^\vee(F)
	\begin{pmatrix}
		z_1\\
		\vdots\\
		z_n
	\end{pmatrix}.
	\]
	Thus the $i$th row of $B^\vee(F)$ is the coefficient vector of
	$\ell_i(F)$ in the fixed dual basis.
	Duality gives
	\[
	B^\vee(F)B(F)=I.
	\]
	Hence
	\(
	B^\vee(F)=B(F)^{-1}.
	\)
	Since $B(F)=I+N(F)$ with $N(F)$ strictly lower triangular and hence
	nilpotent, its inverse is the finite sum
	\[
	B(F)^{-1}
	=
	I-N(F)+N(F)^2-\cdots+(-1)^{n-1}N(F)^{n-1}.
	\]
	Therefore both the adapted frame and its dual coframe depend
	polynomially on the affine coordinates of the big cell.  By construction,
	\(\pi_p(F)\) is spanned by \(v_1(F),\ldots,v_p(F)\). Duality therefore gives
	\[
	\pi_p(F)
	=
	\{\ell_{p+1}(F)=\cdots=\ell_n(F)=0\}.
	\]
	We shall write $L_p=\pi_p(F)$ whenever no confusion can arise.
	\begin{lemma}\label{lem:affine-projection-open}
		A coordinate projection
		$\operatorname{pr}: {\mathbb{C}}^{r+s}\to {\mathbb{C}}^r$ is a Zariski open
		map.
	\end{lemma}
	
	\begin{proof}
		Write the coordinates as $(x,y)\in {\mathbb{C}}^r\times {\mathbb{C}}^s$. Since arbitrary Zariski open sets in $\mathbb C^{r+s}$ are unions of principal open sets, it is enough to consider a principal open set $D(f):=\{f\neq 0\}$ with $f$ a polynomial in $x,y$.  Expand
		$f$ as a polynomial in $y$,
		$f(x,y)=\sum_\alpha c_\alpha(x)y^\alpha$. 
		A point $x$ does not belong to $\operatorname{pr}(D(f))$ exactly when
		$f(x,y)=0$ for every $y$, or equivalently when
		$c_\alpha(x)=0$ for every $\alpha$.  Thus the complement of
		$\operatorname{pr}(D(f))$ is the common zero set of the finitely many
		coefficient polynomials $c_\alpha$.  Hence
		$\operatorname{pr}(D(f))$ is Zariski open.
	\end{proof}

	\begin{proposition}\label{prop:big-cell-grassmannian-projection}
		For every $2\le p\le n-1$, one has
		$\pi_p(\mathcal C)=\mathcal C_p$.  Moreover,
		$\pi_p|_{\mathcal C}:\mathcal C\to\mathcal C_p$ is a Zariski open
		surjection.  Consequently, if
		$\Omega\subset\mathcal C$ is nonempty and Zariski open, then
		$\pi_p(\Omega)$ is a nonempty Zariski open subset of
		$\operatorname{Gr}(p,V)$.
	\end{proposition}
	
	\begin{proof}
		The inclusion $\pi_p(\mathcal C)\subset\mathcal C_p$ follows directly
		from the definition of $\mathcal C$.  Conversely, let
		$L\in\mathcal C_p$.  Since $\operatorname{pr}_p|_L$ is an
		isomorphism, $L$ has a unique graph presentation over $L_p^0$.
		Choosing the corresponding row-echelon basis and completing it by
		Gaussian elimination produces a lower unitriangular matrix of the
		form described above whose first $p$ columns span $L$.  Thus $L$ can
		be completed to a flag in $\mathcal C$, proving
		$\pi_p(\mathcal C)=\mathcal C_p$.
		
		It remains to prove openness.  Write the matrix $B(F)$ in block form
		relative to
		$V=L_p^0\oplus \Span\{e_{p+1},\ldots,e_n\}$:
		
		$$
		B(F)=
		\begin{pmatrix}
			A&0\\
			C&D
		\end{pmatrix}.
		$$
		
		Here $A$ and $D$ are lower unitriangular matrices, while $C$ is
		arbitrary.  The first $p$ columns of
		$B(F)$ form the matrix $\binom{A}{C}$ for $L_p$.  Multiplying these
		columns on the right by $A^{-1}$ does not change their span, so
		$L_p$ is the graph of the linear map whose matrix is
		$X=CA^{-1}$.
		
		Since $A^{-1}$ is polynomial in the entries of $A$, the change of
		variables $(A,C,D)\mapsto(A,X,D)$, with inverse $C=XA$, is  a regular isomorphism (i.e. the entries are all polynomials)
		in both directions.  In these coordinates the big cell is an affine
		space and $\pi_p|_{\mathcal C}$ is simply the coordinate projection
		$(A,X,D)\mapsto X$ onto
		$\mathcal C_p\simeq  \mathbb{C}^{p(n-p)}$.  The assertion now follows
		from Lemma~\ref{lem:affine-projection-open}.
	\end{proof}

	Let $U\subset \mbox{Sym}(V^*)_{\leq D}$ be a finite dimensional linear subspace.
	For a flag
	$F=( L_{n-1}\supset\cdots\supset L_2)$,  let $U_n=U$.
	If $2\le p\le n-1$, write $U_p(F)=U|_{L_p}$, choose a nonzero form $\ell_{p+1}\in L_{p+1}^*$ with
	$L_p=\{ \ell_{p+1}=0\}$ inside $L_{p+1}$. For $0\le j\le D$, put
	\[
	K_p^j(F)
	:=U_{p+1}(F)\cap \ell_{p+1}^j\Sym(L_{p+1}^*).
	\]
	The spaces $K_p^j(F)$ are independent of the nonzero scalar used to normalize
	$\ell_{p+1}$.
	\begin{proposition}
		\label{prop:generic-flag-vanishing}
		Let \(V\) be an \(n\) dimensional complex vector space and let
		\(
		U\subset \Sym(V^*)_{\le D}
		\)
		be a finite dimensional linear subspace.
		Then there exists a nonempty Zariski open subset
		\(
		\Omega_{\mathrm{flag}}(U)\subset \mathcal F(V)
		\)
		such that, for every \(2\le p\le n-1, 0\le j\le D, F\in\Omega_{\mathrm{flag}}(U)\), all the numbers
		\(
		\dim U_p(F),   	
		\)
		and
		\(
		\dim K_p^j(F),
		\)
		are  constant on \(\Omega_{\mathrm{flag}}(U)\). 
	\end{proposition}
	\begin{proof}
		Fix a basis $e_1,\ldots,e_n$ of $V$, with dual basis
		$z_1,\ldots,z_n$ of $V^*$, and order the monomials
		\(z^\alpha\), \(|\alpha|\le D\), by graded lexicographic order. Arrange
		them in a column vector \(\mathcal Z\). Thus \(\mathcal Z\) is a canonical
		basis of the ambient space \(\Sym(V^*)_{\le D}\).
		
		Let \(\mathcal C\subset\mathcal F(V)\) be the  big cell
		determined by this basis. On \(\mathcal C\), every flag \(F\) has a
		unique unitriangular coframe
		\(\ell_1(F),\ldots,\ell_n(F)\) such that
		$$
		L_p(F)=\{\ell_{p+1}(F)=\cdots=\ell_n(F)=0\}.
		$$
		The coefficients of the \(\ell_i(F)\) in the fixed basis
		\(z_1,\ldots,z_n\) are polynomial functions of the affine coordinates
		of \(F\). As explained in the previous section its inverse also have polynomial
		entries. 
		
		For \(|\alpha|\le D\), write
		\(\ell(F)^\alpha=\ell_1(F)^{\alpha_1}\cdots\ell_n(F)^{\alpha_n}\), and
		arrange these monomials in the same order as the \(z^\alpha\) into a
		column vector \(\mathcal L(F)\). There is an invertible matrix \(L(F)\)
		such that
		
		$$
		\mathcal L(F)=L(F)\mathcal Z.
		$$
		
		By the preceding discussion, every entry of \(L(F)\) and
		\(L(F)^{-1}\) is polynomial in cooridinates of  \(\mathcal C\).
		
		We first explain how to place the varying spaces
		\(\Sym(\pi_p(F)^*)_{\le D}\) in one fixed ambient vector space. The
		adapted coframe gives
		\[
		V^*=\Span\{\ell_1(F),\ldots,\ell_p(F)\}
		\oplus\Ann(\pi_p(F)).
		\]
		Hence restriction to \(\pi_p(F)\) identifies \(\pi_p(F)^*\) with
		\(\Span\{\ell_1(F),\ldots,\ell_p(F)\}\subset V^*\), and therefore
		identifies \(\Sym(\pi_p(F)^*)_{\le D}\) with the corresponding
		subspace of \(\Sym(V^*)_{\le D}\).
		
		Now we represent the restriction map $R_p(F) :U \to  \Sym(\pi_p(F)^*)_{\le D},$ it is actually a truncation in adapted basis. 
		Let $X$ be an $N$ dimensional vector space with a fixed basis
		$v_1,\ldots,v_N$, and let
		\[
		\Lambda\subset\{1,\ldots,N\}.
		\]
		We define the truncation associated with $\Lambda$ by
		\[
		T_\Lambda:X\longrightarrow X,\qquad
		T_\Lambda\!\left(\sum_{i=1}^N a_i v_i\right)
		=
		\sum_{i\notin\Lambda}a_i v_i.
		\]
		Thus $T_\Lambda$ simply deletes the components whose indices belong to
		$\Lambda$. The matrix entries of the truncation operator are all either \(0\) or \(1\).

		Under graded lexicographic order basic $\{\ell_1^{\alpha_1} \cdots \ell_p^{\alpha_p}\}_{|\alpha| \le D}$, there exists an index $\Lambda(p)$ that depends only on $p$ such $T_{\Lambda(p)}$ simply
		deletes every monomial containing one of
		\(\ell_{p+1}(F),\ldots,\ell_n(F)\), is represented by a fixed \(0\)-\(1\) truncation
		matrix \(T_p\). 
		
		Then $R_p(F) = T_{\Lambda(p)}|_U $. Choose a basis \(u_1,\ldots,u_m\) of \(U\), and let \(C\) be the fixed
		matrix such that
		\[ \begin{pmatrix}
			u_1 \\
			\vdots\\
			u_m
		\end{pmatrix} = C \mathcal{Z} = C L(F)^{-1} \mathcal{L}(F)
		.
		\]
		Thus the corresponding matrix of $R_p(F)$ under the canonical basis $\mathcal{Z}$ is $$M_p(F) :=C L(F)^{-1} T_p L(F)$$ and $\dim U_p(F) = \rank(M_p(F)). $
		
		We next consider the vanishing filtration, notes that $K_p^0(F)= U_{p+1}(F)$ , we only need to consider the order $j \geq 1$. Inside \(L_{p+1}(F)\), the
		hyperplane \(L_p(F)\) is defined by
		\(\ell_{p+1}(F)|_{L_{p+1}(F)}\). 
		In the monomial basis
		\(\{ \ell_1(F)^{\alpha_1}\ldots\ell_{p+1}(F)^{\alpha_{p+1}} \}_{|\alpha|\le D}\),  consider the truncation map
		\[ R_{p+1,j}(F) : 
		\Sym(\pi_{p+1}(F)^*)_{\le D}
		\longrightarrow
		\Sym(\pi_{p+1}(F)^*)_{\le D}
		,
		\]
		which  simply keeps those monomials for
		whose exponent of \(\ell_{p+1}(F)\) is strictly smaller than \(j\)
		and deletes those divisible by \(\ell_{p+1}(F)^j\). It is therefore
		represented by a \(0\)-\(1\) truncation matrix that depends only on $p,j$ which we
		denote by \(T_{p,j}\). 
		The corresponding matrix of $R_{p+1,j}(F)|_{U_{p+1}(F)}$ under the canonical basis $\mathcal{Z}$ is $$N_{p,j}(F) :=C L(F)^{-1} T_{p+1}T_{p,j} L(F).$$ 
		By definition,
		\(K_p^j(F)\) is precisely the kernel of the  map $R_{p+1,j}(F)$ after
		restriction to  \(U_{p+1}(F)\). Hence rank-nullity gives
		\[
		\dim K_p^j(F)
		=
		\operatorname{rank}M_{p+1}(F)
		-
		\operatorname{rank}N_{p,j}(F).
		\]
		For a matrix with polynomial
		entries, the locus on which its rank is maximal is a nonempty Zariski
		open subset. Apply this to the finitely many matrices \(M_p(F)\) and
		\(N_{p,j}(F)\) occurring above. Since the big cell \(\mathcal C\) is
		irreducible, the intersection of all these maximal-rank loci is again
		a nonempty Zariski open subset of \(\mathcal C\). Denote it by
		\(\Omega_{\mathrm{flag}}(U)\). As \(\mathcal C\) is itself Zariski open in
		\(\mathcal F(V)\), so is \(\Omega_{\mathrm{flag}}(U)\).
		
		On \(\Omega_{\mathrm{flag}}(U)\), all the ranks of the matrices \(M_p(F)\) and
		\(N_{p,j}(F)\) are constant. Therefore all the dimensions
		\(\dim U_p(F)\) and \(\dim K_p^j(F)\) are constant there. 
	\end{proof}
	Let
	$S=\Sym(V^*)$, and let $U\subset S_{\le D}$ be a finite dimensional
	linear subspace of polynomials. For $[\ell] \in \mathbb{P}(V^*)$, define
	\[
	K^j([\ell] )
	:=U  \cap \ell^j\Sym(V^*)  .
	\]
	The notation \(K^j([\ell])\) will be used only for a single moving hyperplane, whereas \(K_p^j\) denotes the corresponding filtration along a flag.
	\begin{lemma}\label{lem:generic-kernel-filtration}
		There exists a nonempty Zariski open subset
		$\Omega(U)\subset\PP(V^*)$ such that, for every $0 \le j\le D$, the
		dimension of $K^j([\ell])$ is constant for $[\ell]\in\Omega(U)$.
		Moreover, near every point of $\Omega(U)$, one can choose a basis of
		each $K^j([\ell])$ whose coefficients depend regularly on the
		coefficients of $\ell$. 
	\end{lemma}
	
	\begin{proof}
		Choose a linear complement \(U'\) of \(U\) in \(S_{\le D}\), and let
		\(\pi:S_{\le D}\to U'\) be the corresponding projection. Fix
		\(0\le j\le D\). On an affine chart of \(\PP(V^*)\), choose the standard
		representative \(\ell\) of \([\ell]\) obtained by setting one fixed nonzero coordinate equal to \(1\), and consider
		\[
		M_j(\ell):S_{\le D-j}\longrightarrow U',
		\qquad
		q\longmapsto\pi(\ell^jq).
		\]
		In fixed bases, the entries of \(M_j(\ell)\) (we also use it to denote the matrix associated to the map $M_j(\ell)$) are polynomial functions
		of the affine coordinates of \([\ell]\). 
		Since multiplication by the nonzero polynomial \(\ell^j\) is injective,
		this identifies \(\ker M_j(\ell)\) with
		\(K^j([\ell])\).
		The rank of a matrix with regular entries is maximal on a nonempty
		Zariski open subset. Hence there is a nonempty Zariski open subset
		\(\Omega_j\subset\PP(V^*)\) on which \(\rank M_j(\ell)\), and therefore
		\(\dim K^j([\ell])\), is constant. Since only finitely many values of
		\(j\) occur, the intersection
		\[
		\Omega(U):=\bigcap_{j=0}^{D}\Omega_j
		\]
		is again nonempty and Zariski open.
		
		It remains to prove the local regularity. Fix
		\([\ell_0]\in\Omega(U)\) and \(j\). Let \(r=\rank M_j(\ell_0)\).
		After choosing bases, select an \(r\times r\) minor that is nonzero at
		\(\ell_0\). Shrinking to a Zariski open neighborhood of \([\ell_0]\),
		this minor remains nonzero, and we may write
		\[
		M_j(\ell)=
		\begin{pmatrix}
			A(\ell)&B(\ell)\\
			C(\ell)&D(\ell)
		\end{pmatrix},
		\qquad
		\det A(\ell)\neq0.
		\]
		On this neighborhood the rank is constantly \(r\), so 
		elimination expresses the pivot coordinates of every element of
		\(\ker M_j(\ell)\) regularly in terms of the free coordinates.
		Equivalently, choosing the free coordinates successively to be the
		standard basis vectors gives a basis
		\(q_1(\ell),\ldots,q_{e_j}(\ell)\) of \(\ker M_j(\ell)\) whose
		coefficients are regular functions of \([\ell]\). Therefore
		\[
		F_i(\ell):=\ell^jq_i(\ell),\qquad 1\le i\le e_j,
		\]
		form a regularly varying local frame of \(K^j([\ell])\).
		Intersecting the finitely many resulting neighborhoods gives such
		frames simultaneously for all \(j\).
	\end{proof}
	\begin{definition}\label{def:general-hyperplane}
		Let
		\(U\subset \Sym(V^*)_{\le D}\) be a finite dimensional polynomial
		space. 
		A point \([\ell]\in\mathbb P(V^*)\), or equivalently the hyperplane \(\ker\ell \subset V\), is said to be \emph{general for \(U\)}
		if
		\[
		\dim K^j([\ell])=\min_{[\lambda] \in \mathbb{P}(V^*)}
		\dim K^j([\lambda]),
		\quad \text{for every }0\le j\le D.
		\]
		We say that a \(p\)-plane \(L\subset V\) is \emph{general for \(U\)} if \(\dim(U|_L)\) takes its  maximal value on \(\operatorname{Gr}(p,V)\).
	\end{definition}
	
	\begin{corollary}\label{cor:local-frame}
		Suppose that \(j\ge1\). Let
		\([\ell_0]\in\mathbb P(V^*)\) be general for \(U\), and let
		\(0\neq F_0\in K^j([\ell_0])\).
		Then one can find a Zariski open neighborhood
		\(\Omega_0\subset\mathbb P(V^*)\) of \([\ell_0]\), a regular choice
		\([\ell]\mapsto\ell\in V^*\setminus\{0\}\) of representative on
		\(\Omega_0\), and a regularly varying family
		\[
		F_{[\ell]}=\ell^j q_{[\ell]}\in U,
		\qquad [\ell]\in\Omega_0,
		\]
		such that \(F_{[\ell_0]}=F_0\).
	\end{corollary}
	
	\begin{proof}
		Choose \(x_0\in V\) with \(\ell_0(x_0)\neq0\). On the affine chart
		\(D(x_0)=\{[\ell]:\ell(x_0)\neq0\}\), there is a unique representative
		\(\ell\) for $[\ell]$ satisfying \(\ell(x_0)=1\). By
		\cref{lem:generic-kernel-filtration}, after shrinking to a Zariski open
		neighborhood \(\Omega_0\subset D(x_0)\cap\Omega(U)\) of \([\ell_0]\),
		there is a regular local frame
		\(q_1([\ell]),\ldots,q_r([\ell])\) of
		\(\ker M_j(\ell)\). Write
		\[
		F_0=\sum_{a=1}^r c_a\ell_0^j q_a([\ell_0])
		\]
		and define
		\[
		F_{[\ell]}:=\sum_{a=1}^r c_a\ell^j q_a([\ell]),
		\qquad [\ell]\in\Omega_0.
		\]
		Then \(F_{[\ell]}\in K^j([\ell])\) varies regularly with
		\([\ell]\) and satisfies \(F_{[\ell_0]}=F_0\).
	\end{proof}

	\section{Admissible Flags }\label{sec: admi flag}
	
	In this section, we further study flags and identify the "good" flags that are most suitable for our study of the restriction problem, which we call admissible flags. In particular, we prove that these admissible flags are general in the flag variety.

	\begin{definition}
		Let $U\subset S$ be a finite dimensional
		primitive polynomial space. A flag
		\[
		F=(V=L_n\supset L_{n-1}\supset\cdots\supset L_2)
		\]
		is called \emph{admissible} for $U$ if the following conditions hold:
		
		\begin{enumerate}[(i)]
			\item for every $2\le p\le n-1$, $L_p$ is general for $U$;
			
			\item for every $2\le p\le n-1$, the hyperplane
			$L_p\subset L_{p+1}$ is general for $U|_{L_{p+1}}$;
			\item $U|_{L_p}$ is primitive for every $2\le p\le n$.
		\end{enumerate}
	\end{definition}
	We now prove that restriction to a general plane preserves primitivity .
	\begin{lemma}
		\label{lem:generic-primitivity}
		Let $V$ be an $n$-dimensional complex vector space, let
		$U\subset\mathbb C[V]$ be a finite dimensional primitive polynomial
		space, and let $2\le p\le n$. Then there exists a nonempty Zariski
		open subset
		\[
		\Omega_p\subset\operatorname{Gr}(p,V)
		\]
		such that $U|_L$ is primitive for every $L\in \Omega_p$.
	\end{lemma}
	
	\begin{proof}
		Let
		\[
		Z(U):=\{z\in V:f(z)=0\text{ for every }f\in U\},
		\]
		and write
		$Z(U)=X_1\cup\cdots\cup X_s$
		as a union of irreducible components. Since $U$ is primitive,
		$Z(U)$ has no component of codimension one, hence
		$\dim X_i\le n-2$ for every $i$.
		
		Fix one component $X=X_i$ and put $m=\dim X$. Consider
		\[
		\mathcal I_{X,p}
		=
		\{(x,L)\in(X\setminus\{0\})\times\operatorname{Gr}(p,V):x\in L\}.
		\]
		This is  a local Zariski closed set, hence a constructible set. For fixed $x\ne0$, the $p$-planes containing $x$ are parametrized by
		$\operatorname{Gr}(p-1,V/\mathbb Cx)$, which has dimension
		$(p-1)(n-p)$. Thus
		\[
		\dim\mathcal I_{X,p}=m+(p-1)(n-p).
		\]
		
		Let
		$q:\mathcal I_{X,p}\to\operatorname{Gr}(p,V)$
		be the second projection. If $q$ is not dominant, i.e., its image is not Zariski dense in $\operatorname{Gr}(p,V)$, then a general
		$p$-plane does not meet $X\setminus\{0\}$. In other words, $X\cap L=0$. If $q$ is dominant, i.e., its image is  Zariski dense in $\operatorname{Gr}(p,V)$, then the
		generic fiber dimension theorem gives, on a nonempty Zariski open
		subset of $\operatorname{Gr}(p,V)$,
		\[
		\dim q^{-1}(L)
		=
		m+(p-1)(n-p)-p(n-p)
		=
		m+p-n
		\le p-2.
		\]
		Since
		$q^{-1}(L)=(X\cap L)\setminus\{0\}$,
		it follows that, for general $L$, the intersection $X\cap L$ has dimension either $0$ or $\leq q-2$, i.e., no
		irreducible component of dimension $p-1$ (codimension-one) in $L$. 
		
		There are only finitely many components $X_i$, so the corresponding
		nonempty Zariski open subsets of $\operatorname{Gr}(p,V)$ may be
		intersected. For every $L$ in the resulting open set,
		\[
		Z(U|_L)=Z(U)\cap L
		\]
		has no codimension-one irreducible component in $L$. Hence $U|_L$ is
		primitive.
	\end{proof}
	\begin{corollary}\label{cor:flag-primitivity}
		Let $V$ be an $n$-dimensional complex vector space, and let
		$U\subset\mathbb C[V]$ be a finite dimensional primitive polynomial
		space.  Then there exists a nonempty Zariski open subset $\Omega \subset \mathcal{F}(V)$ such that, for every $F \in \Omega$ 
		the restriction $U|_{\pi_p(F)}$ is primitive for every
		$2\le p\le n$.
	\end{corollary}
	
	\begin{proof}
		For each $2\le p\le n-1$, let
		$\Omega_p\subset\operatorname{Gr}(p,V)$ be the nonempty Zariski
		open subset furnished by ~\cref{lem:generic-primitivity}, so that
		$U|_L$ is primitive for every $L\in\Omega_p$.
		
		Let
		\[
		\pi_p:\mathcal F(V)\longrightarrow\operatorname{Gr}(p,V),
		\qquad
		F\longmapsto L_p,
		\]
		be the natural projection. Since $\pi_p$ is surjective,
		$\pi_p^{-1}(\Omega_p)$ is a nonempty Zariski open subset of
		$\mathcal F(V)$. Hence
		\[
		\Omega=\bigcap_{p=2}^{n-1}\pi_p^{-1}(\Omega_p)
		\]
		is a nonempty Zariski open subset of $\mathcal F(V)$, because
		$\mathcal F(V)$ is irreducible. On every flag in this intersection,
		all the spaces $U|_{L_p}$ are primitive simultaneously. The case
		$p=n$ is the original hypothesis on $U$.
	\end{proof}

	\begin{theorem}
		\label{thm:admissible-flag}
		Let $V$ be an $n$-dimensional complex vector space, with $n\ge3$, and let
		$U_1,  \cdots U_m\subset\operatorname{Sym}(V^*)_{\le D}$ , where $m \geq 1$ , be finite dimensional
		primitive polynomial spaces. Then there exist a nonempty Zariski open subset $\Omega \subset \mathcal{F}(V)$ such that every $F \in \Omega$  is admissible for $U_1,\cdots , U_m.$ 
	\end{theorem}
	\begin{proof}
		It is enough to consider one polynomial space \(U\). Indeed, once the
		assertion is proved for each \(U_\alpha\), the corresponding nonempty
		Zariski open subsets of \(\mathcal F(V)\) have nonempty intersection,
		since \(\mathcal F(V)\) is irreducible.
		
		Fix therefore a finite-dimensional primitive polynomial space
		\(U\subset\Sym(V^*)_{\le D}\). Let
		\(\Omega_1(U)=\Omega_{\mathrm{flag}}(U)\subset\mathcal F(V)\) be the nonempty Zariski open
		subset furnished by \cref{prop:generic-flag-vanishing}, chosen as in
		its proof by requiring all the matrices \(M_p(F)\) and
		\(N_{p,j}(F)\) to have maximal rank.
		
		For \(F\in\Omega_1(U)\), the numbers
		\(\dim(U|_{L_p})\) take their general values. Indeed,
		\(\pi_p(\Omega_1(U))\) is a nonempty Zariski open subset of
		\(\mathrm{Gr}(p,V)\) by \cref{prop:big-cell-grassmannian-projection}, and
		\(\dim(U|_{L_p})\) is maximal there.
		
		Moreover, the same maximal-rank construction gives successive
		generality. For every \(2\le p\le n-1\), writing
		\(M=L_{p+1}\) and \(L_p=\{\ell_{p+1}=0\}\subset M\), the matrices
		\(N_{p,j}(F)\), \(1\le j\le D\), attain their generic maximal ranks
		as the hyprplane in \(M\) varies. $\mathrm{Gr}(p,M)$ is a subvariety of $\mathrm{Gr}(p,V)$	and $L_p \in \pi_p(\Omega_1(U)) \cap \mathrm{Gr}(p,M) $ , thus \(\pi_p(\Omega_1(U)) \cap \mathrm{Gr}(p,M)\) is a nonempty Zariski open subset of
		\(\mathrm{Gr}(p,M)\) by \cref{prop:big-cell-grassmannian-projection}. Hence
		\[
		\dim\bigl(U|_M\cap\ell_{p+1}^j\Sym(M^*)\bigr)
		\]
		takes its minimal value for every \(j\). Thus
		\(L_p\subset L_{p+1}\) is general for \(U|_{L_{p+1}}\).
		
		By \cref{cor:flag-primitivity}, there is also a nonempty Zariski open
		subset \(\Omega_2(U)\subset\mathcal F(V)\) such that
		\(U|_{L_p}\) is primitive for every \(2\le p\le n\).
		Therefore
		\[
		\Omega(U):=\Omega_1(U)\cap\Omega_2(U)
		\]
		is a nonempty Zariski open subset of \(\mathcal F(V)\), and every
		flag in \(\Omega(U)\) is admissible for \(U\).
	\end{proof}

	\section{Restriction Dynamics along an Admissible Flag}  
	\label{sec:restriction-dynamics}
	In this section, we study the restriction dynamics along an admissible flag.
	
	Let \(U\subset S_{\le D}\) be a finite dimensional primitive polynomial
	space. Fix a flag
	\[
	F=(\C^n=L_n\supset L_{n-1}\supset\cdots\supset L_2)
	\in\mathcal F(\C^n)
	\]
	that is admissible for \(U\), whose existence follows from
	\cref{thm:admissible-flag}. For \(2\le p\le n\), write
	\(U_p:=U|_{L_p}\). For \(2\le p\le n-1\), choose a nonzero
	\(\ell_{p+1}\in L_{p+1}^*\) such that
	\(L_p=\{\ell_{p+1}=0\}\subset L_{p+1}\). We suppress the fixed flag
	from the notation and write
	\[
	K_p^j:=K_p^j(F)
	=U_{p+1}\cap \ell_{p+1}^j\Sym(L_{p+1}^*),
	\qquad 0\le j\le D,
	\]
	with \(K_p^{D+1}=0\). Since the flag is admissible, each
	\(L_p\subset L_{p+1}\) is a general hyperplane for \(U_{p+1}\).

	For \(1\le j\le D\), division by \(\ell_{p+1}^j\), followed by
	restriction to \(L_p\), embeds \(K_p^j/K_p^{j+1}\) into
	\(\Sym(L_p^*)\). We denote its image by
	\[
	Q_p^j:=\bigl\{(f/\ell_{p+1}^j)|_{L_p}:f\in K_p^j\bigr\}.
	\]
	This is well defined because the image vanishes precisely when
	\(f/\ell_{p+1}^j\) is divisible by \(\ell_{p+1}\). We call \(Q_p^j\) the
	\(j\)-th residual space. 
	
	\begin{theorem}
		\label{thm:moving-residual-space}
		With the notation above, for every \(2\le p\le n-1\) and every \(1\le j\le D\), one has
		\[
		\Sym^j(L_p^*)Q_p^j\subset U_p.
		\]	    
	\end{theorem}
	
	\begin{proof}
		Let \(\ell_{p+1}\in L_{p+1}^*\) define \(L_p\subset L_{p+1}\). If
		\(Q_p^j=0\), there is nothing to prove, so assume otherwise.
		
		Choose \(x_0\in L_{p+1}\) with \(\ell_{p+1}(x_0)=1\), and set
		\[
		W:=\{\alpha\in L_{p+1}^*:\alpha(x_0)=0\}.
		\]
		Then \(L_{p+1}^*=\mathbb C\ell_{p+1}\oplus W\), and restriction to
		\(L_p\) identifies \(W\) with \(L_p^*\). On the affine chart
		\(\{\ell(x_0)\ne0\}\subset\mathbb P(L_{p+1}^*)\), normalize each
		linear form by \(\ell(x_0)=1\). Thus every form $\ell$ near \(\ell_{p+1}\) is
		uniquely of the form \(\ell_{p+1}+c\), with \(c\in W\). 
		
		Since \(\ell_{p+1}\) is a general defining form for $U_{p+1}$, after shrinking the affine chart to a parameter
		Zariski open	neighborhood in $L_{p+1}^*$, by \cref{cor:local-frame} one can choose regular sections
		\[
		F_a([\ell])=\ell^j q_a([\ell])\in U_{p+1},
		\qquad 1\le a\le r,
		\]
		whose classes form a frame of
		\(K^j([\ell])/K^{j+1}([\ell])\). The \(q_a([\ell])\) may also be chosen
		to depend regularly on \([\ell]\). At \(\ell=\ell_{p+1}\), put
		\[
		\bar q_a:=q_a([\ell_{p+1}])|_{L_p}.
		\]
		By the definition of the residual space, the polynomials
		\(\bar q_1,\ldots,\bar q_r\) form a basis of \(Q_p^j\). It is
		therefore enough to prove
		\[
		\Sym^j(L_p^*)\bar q_a\subset U_p
		\]
		for each \(a\).
		After shrinking
		once more, there is a Zariski open neighborhood \(\mathcal V\subset W\)
		of \(0\) on which all the sections above are defined.
		
		Fix \(a\), and set
		\[
		v_a(c):=q_a([\ell_{p+1}+c]),\qquad
		G_a(c):=(\ell_{p+1}+c)^jv_a(c).
		\]
		Then \(G_a:\mathcal V\to U_{p+1}\) is regular. Choose arbitrary
		\(b_1,\ldots,b_j\in W\), put
		\(\beta(t)=t_1b_1+\cdots+t_jb_j\), and consider
		\[
		\mathcal G_a(t)
		=
		\bigl(\ell_{p+1}+t_1b_1+\cdots+t_jb_j\bigr)^j
		v_a(\beta(t))=\bigl(\ell_{p+1}+\beta(t)\bigr)^j
		v_a(\beta(t)).
		\]
		For \(t\) near \(0\), this polynomial belongs to the fixed linear
		space \(U_{p+1}\). Hence its mixed derivative
		\[
		D\mathcal G_a(0)
		:=
		\left.
		\frac{\partial^j\mathcal G_a}
		{\partial t_1\cdots\partial t_j}
		\right|_{t=0}
		\]
		also belongs to \(U_{p+1}\).
		
		Apply the multivariable Leibniz rule. Every term in which at least
		one derivative falls on \(v_a(\beta(t))\) has at most \(j-1\)
		derivatives falling on the factor
		\(\bigl(\ell_{p+1}+t_1b_1+\cdots+t_jb_j\bigr)^j\). After evaluation at
		\(t=0\), such a term is still divisible by \(\ell_{p+1}\), and therefore
		vanishes upon restriction to \(L_p\). The only surviving term is the
		one in which all \(j\) derivatives fall on the \(j\)-th power.
		Consequently,
		\[
		D\mathcal G_a(0)
		\equiv
		j!\,b_1\cdots b_j\,q_a(\ell_{p+1})
		\pmod{\ell_{p+1}}.
		\]
		Since \(D\mathcal G_a(0)\in U_{p+1}\), restriction to \(L_p\) gives
		\(\bar b_1\cdots\bar b_j\,\bar q_a\in U_p\), where
		\(\bar b_i=b_i|_{L_p}\in L_p^*\).
		
		The restriction map \(W\to L_p^*\) is an isomorphism, so the
		\(\bar b_i\) are arbitrary elements of \(L_p^*\). Their products span
		\(\Sym^j(L_p^*)\). Hence \(\Sym^j(L_p^*)\bar q_a\subset U_p\) for every
		\(a\), and therefore \(\Sym^j(L_p^*)Q_p^j\subset U_p\).
	\end{proof}
	\begin{definition}
		The vanishing order $\nu_p$ of $U$ at the step
		$U_{p+1}\to U_p$ is the largest integer $j \geq 0$ such that
		$K_p^{j}\neq0$.
	\end{definition}
	Hence
	$K_p^{\nu_p+1}=0$. If $\nu_p\ge1$, then the top residual space
	$Q_p^{\nu_p} \cong K_p^{\nu_p}/K_p^{\nu_p+1}$
	is nonzero.
	\begin{proposition}\label{prop:persistence-nonsimplicity}
		If \(K_p^j\ne0\) for some
		\(3\le p\le n-1\), then \(K_{p-1}^j\ne0\). Consequently, the vanishing orders along a
		admissable flag are nonincreasing:
		\[
		\nu_{p-1}\ge \nu_p .
		\]
	\end{proposition}
	\begin{proof}
		Let \(\nu_p\) be the vanishing order at the step
		\(L_p\subset L_{p+1}\), and suppose \(K_p^j\ne0\). Then
		\(j\le \nu_p\), and by definition
		\(Q_p^{\nu_p}\ne0\). Choose a nonzero element
		\(q\in Q_p^{\nu_p}\). By
		\cref{thm:moving-residual-space}, we have
		\[
		q\Sym^{\nu_p}(L_p^*)\subset U_p .
		\]
		For every \(u\in L_p^*\), it follows that
		\(u^{\nu_p}q\in U_p\cap
		u^{\nu_p}\Sym(L_p^*)\). Since the next member of the admissible flag
		is general for \(U_p\), its defining linear form can be chosen as such
		a general \(u\), and hence the restriction to this hyperplane contains
		an element vanishing to order at least \(\nu_p\). Therefore
		\(K_{p-1}^{\nu_p}\ne0\), which implies
		\[
		\nu_{p-1}\ge \nu_p\ge j .
		\]
		Consequently \(K_{p-1}^{j}\ne0\).
	\end{proof}

	\begin{lemma}\label{lem:pullback}
		Let \(U\subset S=\C[z_1,\ldots,z_n]\) be a finite dimensional
		primitive polynomial space, let \(2\le p\le n\), and let \(c\) be the
		restriction dimension of \(U\) on a general \(p\)-plane.
		\begin{enumerate}[(i)]
			\item If \(c<p\), then \(\dim U=c\).
			\item If \(c=p\), then either \(\dim U=p\), or, for a general
			\(p\)-plane \(L_p\), there is a nonzero polynomial
			\(a=a_{L_p}\in\Sym(L_p^*)\) such that
			\[
			U|_{L_p}=aL_p^*.
			\]
			\item If the second alternative in \emph{(ii)} occurs, then
			\(a\) is constant and \(U\subset S_1\).
		\end{enumerate}
	\end{lemma}
	\begin{proof}
		If \(p=n\), then \(c=\dim U\), so all conclusions are immediate.
		Assume henceforth that \(p<n\); in particular, \(n\ge3\).
		For \(p\le m\le n\), let \(c_m\) be the restriction dimension of
		\(U\) on a general \(m\)-plane. Apply
		\cref{thm:admissible-flag}, and choose a flag admissible for
		\(U\):
		\[
		L_p\subset L_{p+1}\subset\cdots\subset L_n=\C^n.
		\]
		For this flag, all these dimensions take their general values, and each
		\(L_{m-1}\subset L_m\) is general for \(U|_{L_m}\). Then
		\[
		c=c_p\le c_{p+1}\le\cdots\le c_n=\dim U.
		\]
		
		Suppose \(\dim U>c\), and let \(m>p\) be minimal with
		\(c_m>c\). Then \(c_{m-1}=c<c_m\), so the restriction from
		\(U|_{L_m}\) to \(U|_{L_{m-1}}\) has a nonzero kernel. Its vanishing
		order is some \(j\ge1\). By \cref{thm:moving-residual-space}, its residual
		space contains a nonzero \(q\in\Sym(L_{m-1}^*)\) such that
		\[
		q\Sym^j(L_{m-1}^*)\subset U|_{L_{m-1}}.
		\]
		Since multiplication by the nonzero polynomial \(q\) is injective,
		\[p \le m-1 \le\dim\Sym^j(L_{m-1}^*) \le c_{m-1}= c. \]
		This contradicts with $c<p$, thus $\dim U\leq c$. Since
		always \(c\le\dim U\), this proves \emph{(i)}.
		
		Now let \(c=p\) and suppose that \(\dim U>p\).  By the same argument, all the inequalities
		above are then equalities. Hence \(m=p+1\) and \(j=1\).
		The inclusion \(qL_p^*\subset U|_{L_p}\) is therefore an equality.
		The admissible flags form a nonempty open subset of the flag variety,
		whose image under the projection to \(\operatorname{Gr}(p,V)\) contains
		a nonempty open subset by
		\cref{prop:big-cell-grassmannian-projection}. Thus the preceding argument
		applies to a general \(p\)-plane, proving \emph{(ii)}.
		
		Finally, admissibility gives that \(U|_{L_p}\) is primitive.
		Thus \(U|_{L_p}=aL_p^*\) forces \(a\) to be constant. For any fixed
		nonzero homogeneous component of a polynomial in \(U\), the locus of
		\(p\)-planes on which its restriction is nonzero is open and dense.
		Intersecting this locus with the open set furnished by \emph{(ii)} shows
		that no element of \(U\) can have a nonzero homogeneous component of
		degree different from one. Hence \(U\subset S_1\), proving \emph{(iii)}.
	\end{proof}
	\begin{definition}
		For an admissible flag \(F\) for \(U\), the restriction
		\(U_{p+1}\to U_p\) is called \emph{simple} if \(K_p^2=0\), and
		\emph{nonsimple} if \(K_p^2\ne0\).
	\end{definition}
	If the restriction
	\(U_{p+1}\to U_p\) is simple, we abbreviate \(Q_p:=Q_p^1\). This
	space is zero exactly when the restriction is injective.
	
	We now return to the hypotheses of \cref{thm:main}. After the common
	factor reduction in Section~2, both \(H\) and \(V_A\) are primitive.
	Apply \cref{thm:admissible-flag} with \(U_1=H\) and \(U_2=V_A\), and fix a flag
	\begin{equation}\label{eq:weak-sos-admissible-flag}
		\C^n=L_n\supset L_{n-1}\supset\cdots\supset L_2
	\end{equation}
	that is admissible for both \(H\) and \(V_A\). Thus every successive
	hyperplane is general for the current restriction of \(H\), while, for
	both \(H\) and \(V_A\), all restriction dimensions take their general
	values and all restrictions remain primitive. Write
	\[
	H_p:=H|_{L_p},\qquad V_{A,p}:=V_A|_{L_p},\qquad
	\rho_p:=\dim H_p,
	\]
	and, for \(2\le p\le n-1\), define the restriction loss
	\[
	\delta_p:=\rho_{p+1}-\rho_p.
	\]
	Thus \(\delta_p\) is the loss in the restriction
	\(H_{p+1}\to H_p\). For \(U=H\), the notation
	\(K_p^j,Q_p^j\), and \(Q_p=Q_p^1\) will always refer to this
	fixed flag.
	
	The next lemma describes what happens when two consecutive restriction
	losses are equal.
	
	\begin{lemma}\label{lem:endpoint-equal-loss}
		Suppose that the restriction \(H_{p+1}\to H_p\) is simple, where
		\(3\le p\le n-1\), and that
		\(\delta_p=\delta_{p-1}=c\). Then
		\[
		\dim V_{A,p}\le c.
		\]
		If \(c \le p\),  then
		either \(\rank A  \le c \) or \(V_A\subset S_1\). 
	\end{lemma}
	
	\begin{proof}
		Since  the restriction \(H_{p+1}\to H_p\) is simple, \(K_p^2=0\),  thus the kernel of this restriction is $K_p^1$. Division by \(\ell_{p+1} \in L_{p+1}^*\), followed by restriction, identifies \(K_p^1\) with \(Q_p\), and
		\[
		\dim Q_p=\delta_p=c.
		\]
		By \cref{thm:moving-residual-space}, \(L_p^*Q_p\subset H_p\). For a
		general \(u\in L_p^*\),
		\[
		uQ_p\subset H_p\cap u\Sym(L_p^*).
		\]
		The right-hand side is the kernel of the next general hyperplane
		restriction \(H_p\to H_{p-1}\), and hence has dimension
		\(\delta_{p-1}=c\). Multiplication by \(u\) is injective, so equality
		holds:
		\[
		H_p\cap u\Sym(L_p^*)=uQ_p.
		\]
		
		For a general parameter \(\xi\) in $\mathbb C^n$, put
		\(u_\xi=L(\,\cdot\,,\xi)|_{L_p}\). The polarized identity \eqref{eq:polarization} gives
		\[
		u_\xi A(\,\cdot\,,\xi)|_{L_p}
		\in H_p\cap u_\xi\Sym(L_p^*)=u_\xi Q_p.
		\]
		Canceling \(u_\xi\ne0\) yields
		\(A(\,\cdot\,,\xi)|_{L_p}\in Q_p\) for general \(\xi\). Membership
		in the fixed finite dimensional space \(Q_p\) is a closed linear
		condition, so it holds for every \(\xi\). Therefore
		\(V_{A,p}\subset Q_p\) and \(\dim V_{A,p}\le c\).
		
		Let \(c'=\dim V_{A,p}\), which is the general restriction dimension
		of \(V_A\) because the flag is admissible for \(V_A\). If \(c<p\),
		then \(c'\le c<p\), and \cref{lem:pullback} gives
		\(\rank A=\dim V_A=c'\le c\). If \(c=p\), the same argument applies
		when \(c'<p\). When \(c'=c=p\), parts \emph{(ii)}--\emph{(iii)} of
		\cref{lem:pullback} give either \(\rank A=c\) or \(V_A\subset S_1\). Combining the two cases yields the stated alternative
		\(\rank A\le c\) or \(V_A\subset S_1\).
	\end{proof}
	
	\begin{lemma}\label{lem:triangular-staircase}
		Fix \(2\le p \le n-1\). Suppose that the restrictions
		\(H_{t+1}\to H_t\) are simple for every
		\(p+1\le t\le n-1\), with no condition on
		\(H_{p+1}\to H_{p}\). Then
		\[
		1 \le \delta_{n-1}\le\delta_{n-2}\le\cdots\le\delta_{p}.
		\]
		If \(n\ge2\kappa\) (this is only possible for $\kappa\geq 3$), then
		\begin{equation}\label{eq:triangular-staircase}
			\delta_t\ge\min\{n-t,\kappa+1\},
			\qquad p\le t\le n-1.
		\end{equation}
	\end{lemma}
	
	\begin{proof}
		The top restriction has nonzero kernel. Indeed, its defining form
		\(\ell_{n-1}\) is general, and the nondegenerate pairing
		\(L(z,\xi)\) identifies it, up to a nonzero scalar, with
		\(L(\,\cdot\,,\xi_1)\) for a general parameter \(\xi_1\). The
		standard element
		\(\ell_{n-1}A(\,\cdot\,,\xi_1)\) is nonzero and belongs to
		\(H_n\cap\ell_{n-1}S\). Hence \(\delta_{n-1}\ge1\).
		
		Fix \(p<t\le n-1\). The restriction
		\(H_{t+1}\to H_t\) is simple, so \(\dim Q_t=\delta_t\), and
		\cref{thm:moving-residual-space} gives \(L_t^*Q_t\subset H_t\). For
		the general form \(u\in L_t^*\) defining the next restriction,
		\[
		uQ_t\subset H_t\cap u\Sym(L_t^*).
		\]
		The two sides have dimensions \(\delta_t\) and \(\delta_{t-1}\),
		so \(\delta_{t-1}\ge\delta_t\). This proves
		monotonicity down to \(\delta_{p}\).
		
		We prove \eqref{eq:triangular-staircase} by downward induction on
		\(t\). The base case is \(\delta_{n-1}\ge1\). Suppose
		\(p<t\le n-1\) and the estimate is known for \(\delta_t\). If
		\(\delta_t\ge \kappa+1\), monotonicity proves the estimate for
		\(\delta_{t-1}\). Otherwise \(\delta_t\le \kappa\), and the induction
		hypothesis gives \(n-t\le \kappa\). If
		\(\delta_{t-1}=\delta_t\), then
		\[
		t\ge n-\kappa\ge \kappa\ge \delta_t.
		\]
		By \cref{lem:endpoint-equal-loss}, either \(\rank A\le \delta_t \le \kappa\) or
		\(V_A\subset S_1\). The first alternative contradicts
		\cref{prop:triangular-rank}. In the second,
		\(A\) has holomorphic degree at most one; by \cref{remark:degree},
		pointwise nonnegativity then forces its coefficient matrix to be positive
		semidefinite, contrary to \(A\notin\mathrm{SOS}_n\).
		Hence
		\(\delta_{t-1}\ge\delta_t+1\ge n-t+1=n-(t-1)\), completing the
		induction.
	\end{proof}
	
	\section{Divisibility Modulo a Hypersurface}\label{sect: divi dim}
	
	The next section  requires a lower bound for $\dim H_p$. When a nonsimple restriction produces a block $q\Sym^j(L_p^*)\subset H_p$ with nonconstant $q$ , we estimate the image of $H_p$  modulo $(q)$.

	\begin{proposition}\label{prop:quotient-divisibility} 
		Let $R=\C[x_1,\ldots,x_p]$, where $p\ge2$, let $0\ne q\in R$ be nonconstant, and set $\bar R=R/(q)$. 
		Let $U\subset\bar R$ be finite-dimensional. If
		$U\cap\ell\bar R\ne0$ for general $[\ell]\in\PP(R_1)$, then
		$\dim U\ge p-1$.
	\end{proposition}

	We need the following lemma.

	\begin{lemma}\label{prop:moving-mod-q}
		For $0\ne\bar f\in\bar R$, define
		$\Sigma_{\bar f}:=	\{[\ell]\in\PP(R_1):\bar f\in\ell\bar R\}$.	 Then one has $\dim_{\mathbb C}\Sigma_{\bar f}\le1$.
	\end{lemma}
	
	\begin{proof}
		Choose a representative $f\in R$ of $\bar f$. Since $\bar f\ne0$, there
		is an irreducible factor $s$ of $q$ for which the $s$-adic order of
		$f$ is smaller than that of $q$. Write
		\[
		f=s^af_1,\qquad q=s^mq_1,
		\]
		where $a<m$ and $s\nmid f_1q_1$. If $[\ell]\in\Sigma_{\bar f}$,
		then $f=qb+\ell c$ for some $b,c\in R$. When $s\nmid\ell$, the
		irreducibility of $s$ gives $\gcd(s,\ell)=1$. The identity first implies
		$s^a\mid c$; after dividing by $s^a$ and reducing modulo $s$, it gives
		\[
		\widetilde f_1\in\widetilde\ell\,(R/(s)),
		\]
		where tildes denote images in $R/(s)$. If instead $s\mid\ell$, then $s$
		is linear and $[\ell]=[s]$, so this case contributes at most one point.
		Define
		\[
		T_{\widetilde f_1}:=
		\{[\ell]\in\PP(R_1):
		\widetilde f_1\in\widetilde\ell\,(R/(s))\}.
		\]
		Apart from the possible single point just described,
		$\Sigma_{\bar f}$ is contained in $T_{\widetilde f_1}$. It therefore
		suffices to prove $\dim T_{\widetilde f_1}\le1$.
		
		Let $X=V(s)\subset\C^p$ and let $\overline X\subset\PP^p$ be its
		projective closure. Let $\nu:Y\to\overline X$ be the normalization and
		put $Y^\circ:=\nu^{-1}(X)$. Then $Y$ is an irreducible normal
		projective variety. We use standard facts about normalization and
		divisors on normal varieties; see, for example, \cite{Hartshorne1977}. Set
		\[
		\mathcal L:=\nu^*\mathcal O_{\overline X}(1).
		\]
		Each $\ell\in R_1$ induces a section $\sigma_\ell\in H^0(Y,\mathcal L)$.
		If $[\ell]\in T_{\widetilde f_1}$ and $\sigma_\ell\ne0$, then
		$f_1=\ell c$ on $X$. Over $Y^\circ$, every prime divisor in the zero
		divisor of $\sigma_\ell$ is therefore contained in the zero divisor of the
		regular function $(f_1\circ\nu)|_{Y^\circ}$. After taking closures in
		$Y$, all other prime-divisor components belong to the finitely many
		codimension-one components of the fixed boundary $Y\setminus Y^\circ$.
		Hence there are finitely many prime divisors $D_1,\ldots,D_t$ that can
		occur in $\operatorname{div}(\sigma_\ell)$ as $[\ell]$ varies in
		$T_{\widetilde f_1}$.
		
		For each $D_i$, the orders of vanishing of the nonzero sections in the
		finite-dimensional linear system
		\[
		W:=\{\sigma_\ell:\ell\in R_1\}\subset H^0(Y,\mathcal L)
		\]
		are uniformly bounded. Indeed, the subspaces
		\[
		W_{i,r}:=\{\sigma\in W:\operatorname{ord}_{D_i}(\sigma)\ge r\},
		\qquad r\ge0,
		\]
		form a descending chain. Its stable value is zero, because a nonzero
		section has finite order along $D_i$. Therefore only finitely many
		effective divisors can occur as $\operatorname{div}(\sigma_\ell)$ for
		$[\ell]\in T_{\widetilde f_1}$.
		
		Two nonzero sections of $\mathcal L$ with the same divisor are scalar
		multiples. Indeed, their ratio is a rational function with neither a zero
		nor a pole along any prime divisor. Normality makes this ratio and its
		inverse regular on $Y$, and every global regular function on the
		irreducible projective variety $Y$ is constant.
		
		Finally, consider the linear map
		\[
		\sigma:R_1\longrightarrow H^0(Y,\mathcal L),\qquad
		\ell\longmapsto\sigma_\ell.
		\]
		Its kernel has dimension at most one. In fact, a nonzero linear form in
		$\ker\sigma$ vanishes on $X=V(s)$ and hence belongs to the prime ideal
		$(s)$; this is possible only when $s$ is linear, in which case the kernel
		is $\C s$. For any divisor occurring above, the nonzero sections in
		$\operatorname{im}\sigma$ having that divisor span at most one line. The
		inverse image of that
		line in $R_1$ therefore has dimension at most
		$\dim\ker\sigma+1\le2$. Since only finitely many divisor types occur,
		$T_{\widetilde f_1}\setminus\PP(\ker\sigma)$ is contained in a finite union
		of projective linear spaces of dimension at most one. The omitted
		projectivized kernel has dimension at most zero. Hence
		$\dim T_{\widetilde f_1}\le1$, and the result follows.
	\end{proof}

	\begin{proof}[Proof of  \cref{prop:quotient-divisibility}]
		Set $r:=\dim U$ and $m:=\deg q$, and let
		$\pi:R\to\bar R$ be the quotient map. Choose a finite-dimensional
		polynomial space $F\subset R$ such that $\pi|_F:F\to U$ is an
		isomorphism, and choose $d$ with $F\subset R_{\le d}$.
		
		We first record a uniform degree bound. If $f\in F$ and
		$\bar f\in\ell\bar R$, then there exist $b,c\in R$ such that
		\begin{equation}\label{eq:bounded-divisibility-lift}
			f=qc+\ell b,\qquad \deg c\le d,\qquad \deg b\le m+d-1.
		\end{equation}
		Indeed, begin with $f=qc_0+\ell b_0$ and reduce modulo $\ell$. Writing
		$h_\ell$ for the image of $h\in R$ in the integral domain
		$R/(\ell)\simeq\C[y_1,\ldots,y_{p-1}]$, we obtain
		$f_\ell=q_\ell(c_0)_\ell$. If $q_\ell=0$ or $f_\ell=0$, then
		$\ell\mid f$, and we may take $c=0$ and $b=f/\ell$. Otherwise
		\[
		\deg(c_0)_\ell=\deg f_\ell-\deg q_\ell\le d.
		\]
		Choose a lift $c\in R_{\le d}$ of $(c_0)_\ell$. Then $f-qc$ is
		divisible by $\ell$. If it is zero, take $b=0$; otherwise
		$b:=(f-qc)/\ell$ has degree at most $m+d-1$. This proves
		\eqref{eq:bounded-divisibility-lift}.
		
		Consider the finite-dimensional incidence set
		\[
		\widetilde{\mathcal I}:=
		\left\{(\ell,f,c,b)\in
		(R_1\setminus\{0\})\times(F\setminus\{0\})\times
		R_{\le d}\times R_{\le m+d-1}:f=qc+\ell b\right\}.
		\]
		It is locally closed in a finite-dimensional affine space. The morphism
		\[
		\Phi:\widetilde{\mathcal I}\longrightarrow
		\PP(R_1)\times\PP(U),\qquad
		(\ell,f,c,b)\longmapsto([\ell],[\bar f]),
		\]
		has constructible image, by Chevalley's theorem \cite{CC55}. The degree
		bound above shows that this image is exactly
		\[
		\mathcal I:=\{([\ell],[u])\in\PP(R_1)\times\PP(U):
		0\ne u\in U\cap\ell\bar R\}.
		\]
		By hypothesis, the first projection of $\mathcal I$ contains a nonempty
		open subset of $\PP(R_1)$; hence $\dim\mathcal I\ge p-1$.
		
		Decompose the constructible set $\mathcal I$ into finitely many
		irreducible locally closed subvarieties, and fix one of them, say $Z$.
		Let
		\[
		T:=\overline{\operatorname{pr}_2(Z)}\subseteq\PP(U).
		\]
		Then $\operatorname{pr}_2|_Z:Z\to T$ is dominant. By the fiber-dimension
		theorem, its general fiber has dimension $\dim Z-\dim T$. For every
		$[u]\in\operatorname{pr}_2(Z)$, choose a nonzero representative
		$u\in U$. The corresponding fiber is contained in
		\[
		\Sigma_u=\{[\ell]\in\PP(R_1):u\in\ell\bar R\},
		\]
		which has dimension at most one by \cref{prop:moving-mod-q}. Therefore
		\[
		\dim Z\le\dim T+1\le\dim\PP(U)+1=r.
		\]
		Taking the maximum over the finitely many pieces gives
		$\dim\mathcal I\le r$. Combining the two bounds on $\dim\mathcal I$
		yields $\dim U=r\ge p-1$.
	\end{proof}

	\section{Proof of \cref{conj: weak sos} for \texorpdfstring{$n=6$ and $n\ge8$}{n = 6 and n >= 8}}\label{sect: big dim}
	
	Throughout the proof of the Weak SOS Conjecture, we use the fixed flag
	\(\C^n=L_n\supset L_{n-1}\supset\cdots\supset L_2\)
	in \eqref{eq:weak-sos-admissible-flag}, which is admissible for both \(H\)
	and \(V_A\). Recall that $\kappa=\kappa(n)$ be the unique integer such that $\binom{\kappa}{2}<n\leq \binom{\kappa+1}{2}$ and $\Lambda_n=n\kappa-\binom{\kappa}{2}-1$. In \cref{lem:triangular-staircase}, it shows that when $n\geq 2\kappa$, there are better estimates on the loss of dimensions under restrictions. In this section, we consider this case $n\geq 2\kappa$  separately. The possible cases are $n=6$ and $n\geq 8$: for \(n=6\), one has
	\(\kappa=3\); for \(8\le n\le10\), one has \(\kappa=4\); and if
	\(\kappa\ge5\), then
	\(n\ge\binom{\kappa}{2}+1\ge2\kappa\). We establish the following theorem.

	\begin{theorem}\label{thm: kappa big}The \cref{conj: weak sos} holds true for  \(n=6\) or \(n\ge8\). 
	\end{theorem}
	
	We first introduce the following definition.
	\begin{definition}
		A nonsimple restriction \(H_{p+1} \xrightarrow{N} H_p\) is called
		\emph{exceptional} if
		\[
		K_p^3=0
		\qquad\text{and}\qquad
		Q_p^2=\C\cdot1.
		\]
	\end{definition}
	When a nonsimple restriction occurs, let \(p_*\) denote the target
	dimension of the first one, thus  the first nonsimple step is
	\[
	H_{p_*+1}\longrightarrow H_{p_*}.
	\]
	Thus,  by \cref{prop:persistence-nonsimplicity} the simple/nonsimple pattern has the form
	\[
	H_n\xrightarrow{S}\cdots\xrightarrow{S}H_{p_*+1}
	\xrightarrow{N}H_{p_*}\xrightarrow{N}\cdots\xrightarrow{N}H_2,
	\]
	with the obvious interpretation, when there are no simple restrictions or non-simple restrictions.
	
	The proof of \cref {thm: kappa big} will be divided into three  cases: 
	\begin{itemize} 
		\item  all restrictions along the
		fixed flag are simple,  this is done in  \cref{prop:all-simple};
		\item the first nonsimple restriction is  nonexceptional, this is done in  \cref{prop:nonexceptional};
		\item the first nonsimple restriction is exceptional, this is done in \cref{prop:exceptional}.
	\end{itemize}
	
	We shall repeatedly use the following orthogonal splitting. 	For further discussion of orthogonality, we refer the reader to
	\cite{GN23,GN24}.
	
	\begin{lemma}
		\label{lem:orthogonal-central-split}
		For every \(2\le p\le n-2\),
		\begin{equation}\label{eq:orthogonal-central-split}
			\rho\ge \rho_p+\rho_{n-p}.
		\end{equation}
	\end{lemma}
	
	\begin{proof}
		Write
		\[
		\mathbf h(z):=(h_1(z),\ldots,h_\rho(z))\in\C^\rho.
		\]
		It is easy to see that $\dim \mbox{Span}\{\mathbf h(z):z\in \mathbb C^n\}=\rho$.
		For any \(p\)-plane \(L\), the annihilator  $\{\mathbf c\in \mathbb C^\rho: \sum_{\mu=1}^\rho c_\mu h_\mu(z)=0 \mbox{~for~ all~} z\in L\}$ of
		\(\Span\{\mathbf h(z):z\in L\}\)  is precisely the kernel of the
		restriction map \(H\to H|_L\), through the map $\mathbf c\mapsto \sum_{\mu=1}^\rho c_\mu h_{\mu}(z)$.  Hence
		\[
		\dim\Span\{\mathbf h(z):z\in L\}=\dim(H|_L).
		\]
		The locus on which the right-hand side equals \(\rho_p\) is Zariski
		open, and therefore Euclidean open and dense, in
		\(\operatorname{Gr}(p,\C^n)\). The orthogonal-complement map is a
		diffeomorphism from \(\operatorname{Gr}(p,\C^n)\) to
		\(\operatorname{Gr}(n-p,\C^n)\). We may therefore choose an orthogonal
		decomposition \(\C^n=L\oplus L^\perp\), with \(\dim L=p\) and $\dim L^\perp=n-p$, such that
		both summands realize their corresponding general restriction
		dimensions for $H$. For \(z\in L\) and
		\(w\in L^\perp\), the polarized identity gives
		\[
		\sum_{\mu=1}^\rho h_\mu	(z)\overline{h_\mu(w)}=:\langle\mathbf h(z),\mathbf h(w)\rangle
		=L(z,\bar w)A(z,\bar w)=0.
		\]
		Thus the two evaluation spans $\Span\{\mathbf h(z):z\in L\}$ and $\Span\{\mathbf h(z):z\in L^\perp\}$ are orthogonal subspaces of
		\(\C^\rho\), of dimensions \(\rho_p\) and \(\rho_{n-p}\),
		respectively. This proves \eqref{eq:orthogonal-central-split}.
	\end{proof}
	We now deal with the case that all the restrictions are simple.
	\begin{proposition}\label{prop:all-simple}
		Assume that \(n=6\) or \(n\ge8\), and that every restriction along
		the fixed flag is simple. Then \(\rho\ge\Lambda_n\).
	\end{proposition}
	
	\begin{proof}
		From \cref{lem:orthogonal-central-split}, we only need to estimate $\rho_\kappa$ and $\rho_{n-\kappa}$.

		Apply \cref{lem:triangular-staircase} with \(p=2\). Since
		\(n-2\ge \kappa+1\) in the present range,
		\(\delta_2\ge \kappa+1\). The bottom residual space \(Q_2\) has
		dimension \(\delta_2\), and
		\cref{thm:moving-residual-space} gives \(L_2^*Q_2\subset H_2\). Hence
		\cref{lem:two-variable} yields
		\(\rho_2\ge\delta_2+1\ge \kappa+2\).
		
		In passing from \(\rho_2\) to \(\rho_\kappa\), the relevant losses are
		\(\delta_2,\ldots,\delta_{\kappa-1}\), all of which are at least \(\kappa+1\)  by \cref{lem:triangular-staircase}, because
		\(n\ge2\kappa\). Thus
		\[
		\rho_\kappa\ge(\kappa+2)+(\kappa-2)(\kappa+1)=\kappa^2.
		\]
		Similarly, the losses used in passing from \(\rho_\kappa\) to
		\(\rho_{n-\kappa}\) are \(\delta_\kappa,\ldots,\delta_{n-\kappa-1}\), so
		\[
		\rho_{n-\kappa}\ge \kappa^2+(n-2\kappa)(\kappa+1).
		\]
		
		Applying \cref{lem:orthogonal-central-split} with \(p=\kappa\) gives
		\(\rho\ge\rho_\kappa+\rho_{n-\kappa}\ge n(\kappa+1)-2\kappa\). Subtracting
		\(\Lambda_n\) yields \(n-2\kappa+\binom \kappa2+1\ge0\), since \(n\ge2\kappa\).
	\end{proof}
	
	Now, we deal with the case that the first nonsimple restriction is nonexceptional.
	Put
	\[
	B_p:=\binom{p+1}{2}+p-1.
	\]
	The key local estimate is the following.
	
	\begin{lemma}\label{lem:nonexceptional-profile}
		Let \(H_{p_*+1}\to H_{p_*}\) be the first nonsimple restriction
		along the fixed flag, and assume that it is nonexceptional. Then, for
		every \(2\le p\le p_*\),
		\[
		\rho_p\ge B_p.
		\]
	\end{lemma}
	
	\begin{proof}
		Because the restriction \(H_{p_*+1}\to H_{p_*}\) is 
		nonexceptional, by definition, we have    either \(\nu_{p_*}\ge3\), or \(\nu_{p_*}=2\) and there is a nonconstant \(q\)  in  \(  Q_{p_*}^{\nu_{p_*}}\).
		By \cref{thm:moving-residual-space},
		\[
		q\Sym^{\nu_{p_*}}(L_{p_*}^*)\subset H_{p_*}.
		\]
		
		Fix \(2\le p\le p_*\), and choose a general \(p\)-plane
		\(P\in \operatorname{Gr}(p,L_{p_*})\). Put \(H_P:=H|_P\),
		\(V_{A,P}:=V_A|_P\), and
		\(q_P:=q|_P\). The plane \(P\) can also be chosen so that both \(H_P\) and \(V_{A,P}\)
		are primitive, and 
		\(q_P\ne0\),   so that \(q_P\) is nonconstant when \(\nu_{p_*}=2\).
		Moreover, \(\dim H_P\le \rho_p\).  Restricting the preceding block gives
		\[
		q_P\Sym^{\nu_{p_*}}(P^*)\subset H_P.
		\]
		If \(\nu_{p_*}\ge3\), multiplication by \(q_P\) is injective and hence
		\[
		\rho_p\ge\binom{p+\nu_{p_*}-1}{\nu_{p_*}}\ge\binom{p+2}{3}\ge B_p,
		\]
		where the last inequality follows from
		\(\binom{p+2}{3}-B_p=(p-1)(p-2)(p+3)/6\).
		
		It remains to treat \(\nu_{p_*}=2\). Let
		\(\varpi_P:\Sym(P^*)\to\Sym(P^*)/(q_P)\) be the quotient map and
		put \(\overline H_P:=\varpi_P(H_P)\). The kernel of
		\(\varpi_P|_{H_P}\) contains \(q_P\Sym^2(P^*)\), and hence has
		dimension at least \(\binom{p+1}{2}\).
		
		We claim that
		\(\overline H_P\cap u(\Sym(P^*)/(q_P))\ne0\) for a general
		\(u\in P^*\). The nondegeneracy of \(L\) makes the map
		\(\xi\mapsto L(\,\cdot\,,\xi)|_P\) surjective onto \(P^*\). We may
		therefore write \(u=L(\,\cdot\,,\xi)|_P\) for a general
		\(\xi\in\C^n\), with \(u\) relatively prime to \(q_P\). The standard
		element \(uA(\,\cdot\,,\xi)|_P\) belongs to
		\(H_P\cap u\Sym(P^*)\), and its image modulo \(q_P\) is nonzero for
		general \(\xi\). Otherwise, coprimality of \(u\) and \(q_P\) would
		imply that \(q_P\) divides \(A(\,\cdot\,,\xi)|_P\) for general
		\(\xi\), hence for every \(\xi\) because divisibility by the fixed
		polynomial \(q_P\) is a closed linear condition. This would make
		\(q_P\) a common factor of \(V_{A,P}\), contrary to primitivity. Thus
		\cref{prop:quotient-divisibility} gives
		\(\dim\overline H_P\ge p-1\), and rank--nullity yields
		\[
		\rho_p\ge\dim H_P\ge\binom{p+1}{2}+p-1=B_p.
		\]
	\end{proof}
	
	\begin{proposition}\label{prop:nonexceptional}
		Assume that \(n=6\) or \(n\ge8\), and that the first nonsimple
		restriction along the fixed flag is nonexceptional. Then
		\(\rho\ge\Lambda_n\).
	\end{proposition}
	
	\begin{proof}
		Let \(p_*\) be the target dimension of the first nonsimple
		restriction and put
		\[
		\mathcal R:=\rho_\kappa+\rho_{n-\kappa}.
		\]
		By \cref{lem:orthogonal-central-split}, it is enough to prove
		\(\mathcal R\ge\Lambda_n\). Recall that, for \(2\le p\le q\le n\),
		\[
		\rho_q=\rho_p+\sum_{i=p}^{q-1}\delta_i.
		\]
		We distinguish the position of \(p_*\) relative to
		\(\kappa\) and \(n-\kappa\).
		
		If \(n-\kappa\le p_*\), both dimensions \(\kappa\) and
		\(n-\kappa\) lie at or below the first nonsimple step. Hence
		\cref{lem:nonexceptional-profile} gives directly
		\(
		\rho_\kappa\ge B_\kappa,\
		\rho_{n-\kappa}\ge B_{n-\kappa},
		\)
		and therefore
		\(
		\mathcal R\ge B_\kappa+B_{n-\kappa}.
		\)
		
		Suppose next that
		\(\kappa\le p_*<n-\kappa\). Then
		\cref{lem:nonexceptional-profile} gives
		\(
		\rho_\kappa\ge B_\kappa,
		\
		\rho_{p_*}\ge B_{p_*}.
		\)
		In this case, the restriction steps from $H_{n-\kappa}$ to $H_{p_*+1}$ are all simple restrictions, \cref{lem:triangular-staircase} gives
		\(\delta_i\ge\kappa+1\) for $p_*\leq i\leq n-k-1$. Thus
		\[
		\rho_{n-\kappa}
		=
		\rho_{p_*}+\sum_{i=p_*}^{n-\kappa-1}\delta_i
		\ge
		B_{p_*}+(n-\kappa-p_*)(\kappa+1),
		\]
		so
		\[
		\mathcal R\ge
		B_\kappa+B_{p_*}
		+(n-\kappa-p_*)(\kappa+1).
		\]
		
		Finally, suppose that \(2\le p_*<\kappa\). Now both
		\(\kappa\) and \(n-\kappa\) lie above the first nonsimple target.
		Starting from
		\(\rho_{p_*}\ge B_{p_*}\), the same  estimate as above gives
		\[
		\rho_\kappa
		\ge
		B_{p_*}+(\kappa-p_*)(\kappa+1)
		\]
		and
		\[
		\rho_{n-\kappa}
		\ge
		B_{p_*}+(n-\kappa-p_*)(\kappa+1).
		\]
		Hence
		\[
		\mathcal R\ge
		2B_{p_*}+(n-2p_*)(\kappa+1).
		\]
		
		We have therefore obtained
		\begin{equation}\label{eq:NS-two-range-count}
			\mathcal R\ge
			\begin{cases}
				B_\kappa+B_{n-\kappa},
				& n-\kappa\le p_*\le n-1,\\[3pt]
				B_\kappa+B_{p_*}
				+(n-\kappa-p_*)(\kappa+1),
				& \kappa\le p_*<n-\kappa,\\[3pt]
				2B_{p_*}+(n-2p_*)(\kappa+1),
				& 2\le p_*<\kappa.
			\end{cases}
		\end{equation}
		
		It remains to minimize these three expressions. Since
		\(
		B_{p+1}-B_p=p+2,
		\)
		the increment of the middle expression when \(p_*\) is replaced
		by \(p_*+1\) is
		\[
		(B_{p_*+1}-B_{p_*})-(\kappa+1)
		=p_*-\kappa+1.
		\]
		This is positive for \(p_*\ge\kappa\), so the middle branch is
		increasing and attains its minimum at \(p_*=\kappa\).
		
		For the third branch, the corresponding increment is
		\[
		2(B_{p_*+1}-B_{p_*})-2(\kappa+1)
		=
		2(p_*-\kappa+1).
		\]
		It is nonpositive for \(p_*\le\kappa-1\), so the third branch is
		nonincreasing and attains its minimum at \(p_*=\kappa-1\).
		Moreover,
		\[
		2B_{\kappa-1}
		+(n-2\kappa+2)(\kappa+1)
		=
		2B_\kappa+(n-2\kappa)(\kappa+1),
		\]
		Thus the two
		adjacent branches have the same minimum. The first branch equals
		the continuation of the middle branch at \(p_*=n-\kappa\), and
		hence cannot give a smaller value. Consequently,
		\[
		\mathcal R
		\ge
		2B_\kappa+(n-2\kappa)(\kappa+1).
		\]
		
		Finally, using the definitions of \(B_\kappa\) and \(\Lambda_n\),
		\[
		2B_\kappa+(n-2\kappa)(\kappa+1)-\Lambda_n
		=
		n-\binom{\kappa}{2}-1.
		\]
		By the defining inequality
		\(\binom{\kappa}{2}<n\), the right hand side is nonnegative.
		Therefore
		\(\mathcal R\ge\Lambda_n\), and
		\cref{lem:orthogonal-central-split} gives
		\(\rho\ge\Lambda_n\).
	\end{proof}
	
	It remains to treat the case that the first nonsimple restriction is exceptional.
	
	\begin{proposition}
		\label{prop:exceptional}
		Assume that \(n=6\) or \(n\ge8\). Suppose that the first nonsimple
		restriction along the fixed flag is exceptional. Then
		\(\rho\ge\Lambda_n\).
	\end{proposition}
	
	\begin{proof}
		Let \(p_*\) be the target dimension of the first nonsimple
		restriction, so that this step is
		\(H_{p_*+1}\to H_{p_*}\), and put \(\ell_0:=\ell_{p_*+1}\). By
		exceptionality, choose \(0\ne w_0\in K_{p_*}^2\) whose residual class
		is a nonzero constant. By \cref{cor:local-frame}, after shrinking to a
		Zariski open neighborhood of \([\ell_0]\), choose the regular
		representatives \(\ell\) supplied there and extend \(w_0\) to a
		regularly varying family
		\[
		w_{[\ell]}=\ell^2 q_{[\ell]}\in H_{p_*+1}.
		\]
		At the base point, the choice of \(w_0\) gives
		\[
		q_{[\ell_0]}=c_0+\ell_0 b_0,
		\qquad c_0\ne0.
		\]
		The constant coefficient \(c([\ell]):=q_{[\ell]}(0)\) varies regularly
		and satisfies \(c([\ell_0])=c_0\ne0\). After shrinking once more, it is
		nonzero throughout the neighborhood. Consequently, the quadratic
		homogeneous component of \(w_{[\ell]}\) is
		\(c([\ell])\ell^2\). As \([\ell]\) ranges over a nonempty
		open subset of \(\PP(L_{p_*+1}^*)\), these pure squares span
		\(\Sym^2(L_{p_*+1}^*)\). Let
		\[
		H^{[2]}:=\Span\{h_{\mu,2}:1\le\mu\le\rho\}
		\]
		be the image of \(H\) under projection onto its quadratic homogeneous
		component. Therefore
		\[
		\Sym^2(L_{p_*+1}^*)\subset H^{[2]}|_{L_{p_*+1}}.
		\]
		
		The bidegree-\((2,2)\) identity
		\eqref{eq:homogeneous-polarized-identity} shows that
		\(\|z\|^2A_{1,1}(z,\bar z)\ge0\). Hence
		\(A_{1,1}(z,\bar z)\ge0\) for \(z\ne0\), and continuity gives the same
		at the origin. Since \(A_{1,1}\) has bidegree \((1,1)\), its Hermitian
		coefficient matrix is positive semidefinite. Choose a minimal decomposition
		\[
		A_{1,1}=\sum_{\alpha=1}^{a_1}|v_\alpha|^2,
		\qquad
		V_{A_{1,1}}:=\Span\{v_1,\ldots,v_{a_1}\}\subset E.
		\]
		Then we have 
		\[
		L(z,\xi)A_{1,1}(z,\xi)
		=\sum_{j,\alpha}(z_jv_\alpha(z))
		(z_jv_\alpha)^*(\xi).
		\]
		Again applying \eqref{eq:homogeneous-polarized-identity}, it holds that 
		\[\sum_{j,\alpha}(z_jv_\alpha(z))
		(z_jv_\alpha)^*(\xi)=\sum_{\mu=1}^\rho h_{\mu,2}(z)h_{\mu,2}^*(\xi).\]
		From \cite[Chapter3, Proposition 3]{Dan93}, this imples that 
		\[
		H^{[2]}=EV_{A_{1,1}}.
		\]
		Recall that $E=S_1$ is the space of linear forms on $\mathbb C^n$.
		This identity also appears in a more general form in \cite{Dan05}.
		
		Put
		\(V_{A_{1,1},p_*+1}:=V_{A_{1,1}}|_{L_{p_*+1}}\). Restriction of
		the preceding inclusion gives
		\[
		\Sym^2(L_{p_*+1}^*)
		\subset L_{p_*+1}^*V_{A_{1,1},p_*+1}.
		\]
		This forces \(L_{p_*+1}^*=V_{A_{1,1},p_*+1}\). 
		Consequently, \(p_*+1 \le \dim V_{A_{1,1}}=a_1\le n\). By
		\cref{lem:filtered-macaulay},
		\[
		\dim(EV_{A_{1,1}})\ge na_1-\binom{a_1}{2}.
		\]
		Homogeneous projection cannot increase dimension, so
		\(\dim H^{[2]}\le\dim H=\rho\). Since
		\(x\mapsto nx-\binom x2\) is increasing for \(1\le x\le n\), we get
		\begin{equation}\label{eq:NS-quadratic-global}
			\rho\ge\dim H^{[2]}
			\ge na_1-\binom{a_1}{2}
			\ge n(p_*+1)-\binom{p_*+1}{2}.
		\end{equation}
		It is worth mentioning that the proof  of \eqref{eq:NS-quadratic-global} does not require the
		assumption that \(n\ge2\kappa\), and we shall use it again in the next
		section.
		
		If \(p_*\ge \kappa-1\), this already proves the desired estimate, since
		\eqref{eq:NS-quadratic-global} gives
		\(\rho\ge n\kappa-\binom \kappa2=\Lambda_n+1\).
		
		It remains to treat \(p_*\le \kappa-2\). Since \(n\ge2\kappa\), all restriction steps  from
		dimension \(p_*+1\) up to dimensions \(\kappa\) and \(n-\kappa\) are simple.  Hence \(\delta_i\ge\kappa+1\) for $p_*+1\leq i\leq \min\{\kappa-1,n-\kappa-1\}$. Thus
		\[
		\rho_\kappa\ge\rho_{p_*+1}+(\kappa-p_*-1)(\kappa+1),
		\]
		and
		\[
		\rho_{n-\kappa}\ge\rho_{p_*+1}+(n-\kappa-p_*-1)(\kappa+1).
		\]
		Adding these estimates gives
		\[
		\mathcal R=\rho_\kappa+\rho_{n-\kappa}
		\ge2\rho_{p_*+1}+(n-2p_*-2)(\kappa+1).
		\]
		It remains to estimate $\rho_{p_*+1}$.
		The preceding restriction
		\(H_{p_*+2}\to H_{p_*+1}\) is simple, thus 
		\(\dim Q_{p_*+1}=\delta_{p_*+1}\). Moreover,
		\cref{lem:triangular-staircase} gives
		\(\delta_{p_*+1}\ge \kappa+1\).  By
		\cref{thm:moving-residual-space},
		\(L_{p_*+1}^*Q_{p_*+1}\subset H_{p_*+1}\). Since
		\(p_*\le \kappa-2\), we have
		\(\delta_{p_*+1}\ge \kappa+1\ge p_*+3\). Choose a
		\((p_*+2)\)-dimensional subspace \(Q'\subset Q_{p_*+1}\). By
		\cref{lem:filtered-macaulay}, applied in the \(p_*+1\) variables of
		\(L_{p_*+1}\),
		\[
		\rho_{p_*+1}\ge\dim(L_{p_*+1}^*Q')
		\ge\gamma_{p_*+1}(p_*+2)=
		\binom{p_*+2}{2}+p_*=B_{p_*+1}.
		\]
		Thus
		\[
		\mathcal R\ge
		2B_{p_*+1}+(n-2p_*-2)(\kappa+1).
		\]
		For \(r\le \kappa-1\), the quantity
		\(2B_r+(n-2r)(\kappa+1)\) is nonincreasing as \(r\) increases, because its
		increment from \(r\) to \(r+1\) is \(2(r-\kappa+1)\). Since
		\(p_*+1\le \kappa-1\), we obtain
		\(\mathcal R\ge2B_\kappa+(n-2\kappa)(\kappa+1)\). Finally,
		\(2B_\kappa+(n-2\kappa)(\kappa+1)-\Lambda_n
		=n-\binom \kappa2-1\ge0\) by the definition of \(\kappa\). Therefore
		\(\rho\ge\mathcal R\ge\Lambda_n\) by
		\cref{lem:orthogonal-central-split}. The proof is thus complete.
	\end{proof}

	\section{Proof of \cref{conj: weak sos} for $n=3, 4, 5, 7$}\label{sect: small dim}
	
	In this section, we complete the proof of \cref{thm:main}. In view of \cref{thm: kappa big},	it remains to treat the   cases that $n=3,4,5,7$.  The required lower bounds are
	\(\rho\ge4,8,11,21\) for \(n=3,4,5,7\), respectively.
	
	We retain the fixed flag \eqref{eq:weak-sos-admissible-flag} and all the
	notation of Section~\ref{sec:restriction-dynamics}. Recall that, if
	\(H_{p+1}\to H_p\) is simple, then
	\[
	K_p^1=H_{p+1}\cap\ell_{p+1}\Sym(L_{p+1}^*),\qquad K_p^2=0,
	\]
	the residual space \(Q_p=Q_p^1\) has dimension \(\delta_p\), and
	\cref{thm:moving-residual-space} gives \(L_p^*Q_p\subset H_p\).
	
	We shall repeatedly use
	\cref{lem:endpoint-equal-loss,lem:nonexceptional-profile} and the orthogonal splitting \cref{lem:orthogonal-central-split}. Their proofs do not use $n \ge 2\kappa$. We also use the 
	following dimension estimate obtained in the exceptional branch argument.
	
	\begin{lemma}\label{lem:endpoint-exceptional}
		Suppose that the first nonsimple restriction of \(H\) along the fixed
		flag is \(H_{p_*+1}\to H_{p_*}\), and assume that it is exceptional.
		Then
		\[
		\rho\ge n(p_*+1)-\binom{p_*+1}{2}.
		\]
	\end{lemma}
	
	\begin{proof}
		This is exactly \eqref{eq:NS-quadratic-global}. Its derivation uses
		only the bidegree-\((2,2)\) identity and the filtered Macaulay estimate,
		not the  assumption \(n\ge2\kappa\).
	\end{proof}
	\addtocontents{toc}{\protect\setcounter{tocdepth}{-2}}
	\subsection{Dimension $3$.}
	
	\addtocontents{toc}{\protect\setcounter{tocdepth}{1}}
	Here the required bound is \(\rho\ge4\), while
	\cref{prop:triangular-rank} gives \(\rank A\ge3\). There is only one
	restriction, \(H_3\to H_2\). If it is nonsimple and nonexceptional,
	\cref{lem:nonexceptional-profile} gives \(\rho_2\ge B_2=4\). If it is
	exceptional, \cref{lem:endpoint-exceptional} with \(p_*=2\) gives
	\(\rho\ge6\). Thus it remains only to consider the simple case.
	
	Suppose for contradiction that \(\rho\le3\).  From \cref{lem:triangular-staircase}, we have
	\(\delta_2\ge1\), so \(\rho_2\le2\). A nonzero polynomial remains
	nonzero on a general two-plane, hence \(\rho_2\ge1\). If
	\(\rho_2=1\), the first part of \cref{lem:pullback}, applied to \(H\),
	gives \(\rho=1\), contradicting \(\delta_2\ge1\). Therefore
	\(\rho_2=2\), and then necessarily \(\rho=3\).
	
	Now \(H\) is primitive and its restriction to a general two-plane has
	dimension exactly two. Since \(\dim H=3>2\), the last part of
	\cref{lem:pullback} gives \(H\subset S_1\). The polarized identity then
	forces \(A\) to be constant, contradicting \(\rank A\ge3\).
	Hence \(\rho\ge4\). This completes the proof for dimension $3$.
	
	\addtocontents{toc}{\protect\setcounter{tocdepth}{-2}}
	\subsection{Dimension $4$.}
	\addtocontents{toc}{\protect\setcounter{tocdepth}{1}}
	
	Here the required bound is \(\rho\ge8\), and
	\cref{prop:triangular-rank} gives \(\rank A\ge4\). If the first
	nonsimple restriction is nonexceptional,
	\cref{lem:nonexceptional-profile} gives \(\rho_2\ge B_2=4\). The
	orthogonal decomposition \(2+2\) and
	\eqref{eq:orthogonal-central-split} then give \(\rho\ge2\rho_2\ge8\).
	If the first nonsimple restriction is exceptional, then
	\(p_*\in\{3,2\}\), and \cref{lem:endpoint-exceptional} gives
	\(\rho\ge10\) or \(\rho\ge9\). Thus a counterexample with
	\(\rho\le7\) would have both restrictions simple.
	
	For two simple restrictions, \cref{lem:triangular-staircase} implies \(\delta_3\ge1\) and
	\(\delta_2\ge\delta_3\). If \(\delta_2=1\), then
	\(\delta_3=\delta_2=1\), so \cref{lem:endpoint-equal-loss} with
	\(p=3\) would give \(\rank A\le1\), a contradiction. Hence
	\(\delta_2\ge2\). The residual space \(Q_2\) has dimension
	\(\delta_2\), and by \cref{thm:moving-residual-space} \(L_2^*Q_2\subset H_2\). By
	\cref{lem:two-variable}, \(\rho_2\ge\delta_2+1\ge3\). On the other
	hand, \(\rho\le7\) and the orthogonal \(2+2\) decomposition imply
	\(\rho_2\le3\). Consequently \(\rho_2=3\), \(\delta_2=2\), and
	\(H_2=L_2^*Q_2\).
	
	We need the following elementary lemma. 
	\begin{lemma}\label{lem:two-generator-equality}
		Let \(S=\C[z_1,\ldots,z_m]\), with \(m\ge2\), and let
		\(U=\Span\{f,g\}\subset S\) be two-dimensional. If
		\(\dim(S_1U)=2m-1\), then \(U=aL\), where \(a\) is a nonzero
		polynomial and \(L\subset S_1\) is two-dimensional.
	\end{lemma}
	
	\begin{proof}
		The kernel of the multiplication map \(S_1\otimes U\to S_1U\) is
		one-dimensional, so there is a nonzero linear syzygy
		\(\alpha f+\beta g=0\). Neither coefficient can vanish, and
		\(\alpha\) and \(\beta\) cannot be proportional because \(f\) and
		\(g\) are linearly independent. Thus the two linear forms are
		relatively prime. From \(\alpha f=-\beta g\) and Euclid's lemma,
		\(\beta\) divides \(f\) and \(\alpha\) divides \(g\). Writing
		\(f=a\beta\) and \(g=-a'\alpha\) and substituting back gives
		\(a=a'\). Hence \(U=a\Span\{\beta,-\alpha\}\).
	\end{proof}
	
	Applying the lemma to \(Q_2\), with \(m=2\), gives \(Q_2=aL_2^*\), and
	hence \(H_2=a\Sym^2(L_2^*)\). The restriction \(H_2\) is primitive, so
	\(a\) is constant and \(H_2\subset\Sym^2(L_2^*)\). By  \cref{thm:admissible-flag}, a general two-plane can be extended to an admissible flag. The same argument apply to those flag , so
	\(H|_{L_2}\subset\Sym^2(L_2^*)\) for a general two-plane \(L_2\). If
	some \(f\in H\) had a nonzero homogeneous component of degree different
	from two, that component would remain nonzero on a general two-plane, a
	contradiction. Thus \(H\subset S_2\). The polarized identity now forces
	\(A\) to have bidegree \((1,1)\). Since \(A\) is pointwise nonnegative,
	the asssociated Hermitian matrix of $A$ is positive semidefinite, contrary to
	\(A\notin\mathrm{SOS}_4\). Therefore \(\rho\ge8\). This completes the proof for dimension $4$.
	
	\addtocontents{toc}{\protect\setcounter{tocdepth}{-2}}
	\subsection{Dimension $5$.}
	\addtocontents{toc}{\protect\setcounter{tocdepth}{1}}
	Here the required bound is \(\rho\ge11\), and
	\cref{prop:triangular-rank} gives \(\rank A\ge4\). Suppose first that
	all three restrictions are simple. \cref{lem:triangular-staircase} gives
	\(\delta_4\le\delta_3\le\delta_2\) and \(\delta_4\ge1\). Equality
	\(\delta_4=\delta_3=1\) is excluded by
	\cref{lem:endpoint-equal-loss}, since it would give
	\(\rank A\le1\). Thus \(\delta_3\ge2\). Likewise, if
	\(\delta_2=2\), then \(\delta_3=\delta_2=2\), and the same lemma, now
	with \(p=3\), would give \(\rank A\le2\). Hence
	\(\delta_2\ge3\).
	
	The bottom simple restriction therefore has a residual space \(Q_2\) of
	dimension at least three, with \(L_2^*Q_2\subset H_2\). By
	\cref{lem:two-variable}, \(\rho_2\ge4\), and then
	\(\rho_3=\rho_2+\delta_2\ge7\). Applying
	\eqref{eq:orthogonal-central-split} to an orthogonal decomposition
	\(2+3\) gives \(\rho\ge\rho_2+\rho_3\ge11\).
	
	Now suppose the first nonsimple restriction is nonexceptional. If
	\(p_*\ge3\), then \cref{lem:nonexceptional-profile} gives
	\(\rho_2\ge B_2=4\) and \(\rho_3\ge B_3=8\), hence
	\(\rho\ge\rho_2+\rho_3\ge12\). It remains to consider \(p_*=2\).
	Then  \cref{lem:nonexceptional-profile} gives \(\rho_2\ge4\). The restrictions to
	dimensions four and three are simple. Then $1\leq \delta_4\leq \delta_3\leq \delta_2$. From \(\delta_4\ge1\) and \cref{lem:endpoint-equal-loss}, \(\delta_3\ge2\), otherwise, rank $A\leq 1$.  If \(\delta_2=2\), then
	\(\delta_3=\delta_2=2\), and
	\cref{lem:endpoint-equal-loss} with \(p=3\) would give
	\(\rank A\le2\). Thus \(\delta_2\ge3\), whence
	\(\rho_3=\rho_2+\delta_2\ge7\) and again
	\(\rho\ge\rho_2+\rho_3\ge11\).
	
	Finally, if the first nonsimple restriction is exceptional, then
	\(p_*\in\{4,3,2\}\). The smallest value supplied by
	\cref{lem:endpoint-exceptional} occurs at \(p_*=2\) and is
	\(5\cdot3-\binom32=12\). Thus the exceptional branch also satisfies
	the desired bound. This completes the proof for dimension $5$.
	
	\addtocontents{toc}{\protect\setcounter{tocdepth}{-2}}
	\subsection{Dimension $7$.}
	\addtocontents{toc}{\protect\setcounter{tocdepth}{1}}
	
	Here the required bound is \(\rho\ge21\), and
	\cref{prop:triangular-rank} gives \(\rank A\ge5\).
	
	We first record a consequence that will eliminate the only borderline
	configuration. Suppose that the restrictions \(H_7\to H_6 \to H_5 \to H_4\) are simple. Starting from
	\(\delta_6\ge1\), 
	\cref{lem:endpoint-equal-loss} and  \cref{lem:triangular-staircase}  successively exclude
	\(\delta_6=\delta_5=1\), \(\delta_5=\delta_4=2\), and
	\(\delta_4=\delta_3=3\), since these would imply
	\(\rank A\le1,2,3\), respectively. Hence
	\[
	\delta_6\ge1,\qquad
	\delta_5\ge2,\qquad
	\delta_4\ge3,\qquad
	\delta_3\ge4.
	\]
	
	We claim that, under the same simplicity assumption,
	\[
	(\rho_3,\rho_4)\neq(8,12).
	\]
	Indeed, suppose that \(\rho_3=8\) and \(\rho_4=12\). Then
	\(\delta_3=4\). Since \(3\le\delta_4\le4\), if
	\(\delta_4=4\), then \(\delta_4=\delta_3=4\), and
	\cref{lem:endpoint-equal-loss} with \(p=4\) gives either
	\(\rank A\le4\), contradicting \(\rank A\ge5\), or
	\(V_A\subset S_1\). The latter is also impossible, since the
	pointwise nonnegativity of \(A\) would then make \(A\) a sum of
	squared norms. Thus \(\delta_4=3\).
	
	Now \(2\le\delta_5\le3\). Equality \(\delta_5=3\) would give
	\(\delta_5=\delta_4=3\), and
	\cref{lem:endpoint-equal-loss} would imply \(\rank A\le3\).
	Hence \(\delta_5=2\). Similarly, \(1\le\delta_6\le2\), and
	\(\delta_6=2\) would give \(\delta_6=\delta_5=2\) and hence
	\(\rank A\le2\). Therefore \(\delta_6=1\). It follows that
	\[
	\rho=\rho_7
	=\rho_4+\delta_4+\delta_5+\delta_6
	=12+3+2+1=18.
	\]
	On the other hand, the orthogonal \(3+4\) decomposition gives
	\(\rho\ge\rho_3+\rho_4=20\), a contradiction. This proves the claim.
	
	If all restrictions are simple, then
	\(\delta_2\ge\delta_3\ge4\). The bottom residual space \(Q_2\)
	satisfies \(\dim Q_2=\delta_2\) and \(L_2^*Q_2\subset H_2\).
	By \cref{lem:two-variable}, \(\rho_2\ge\delta_2+1\ge5\).
	Consequently \(\rho_3\ge9\) and \(\rho_4\ge13\), so
	\cref{lem:orthogonal-central-split} gives
	\(\rho\ge\rho_3+\rho_4\ge22\).

	We now distinguish the position of the first nonsimple restriction. 	Let \(p_*\) be the target dimension of the first nonsimple
	restriction, so that this step is
	\(H_{p_*+1}\to H_{p_*}\).
	
	Suppose first that \(p_*\ge4\). If the first nonsimple restriction is
	nonexceptional, then \cref{lem:nonexceptional-profile} gives
	\(\rho_3\ge B_3=8\) and \(\rho_4\ge B_4=13\). Hence the orthogonal
	\(3+4\) decomposition yields
	\(\rho\ge\rho_3+\rho_4\ge21\).
	If the first nonsimple restriction is exceptional, then
	\cref{lem:endpoint-exceptional} gives
	\[
	\rho\ge 7(p_*+1)-\binom{p_*+1}{2}\ge25>21.
	\]
	Thus the desired estimate holds whenever \(p_*\ge4\).
	
	It remains to consider \(p_*\le3\), together with the case in which
	all restrictions are simple. In all these configurations the restrictions
	\(H_7\to H_6 \to H_5 \to H_4\) are simple, so the
	estimates above give \(\delta_3\ge4\).

	Assume now that \(p_*=3\). If the first nonsimple restriction is
	exceptional, then \cref{lem:endpoint-exceptional} gives
	\[
	\rho\ge7\cdot4-\binom42=22.
	\]
	If it is nonexceptional, then
	\cref{lem:nonexceptional-profile} gives \(\rho_3\ge8\), while
	\(\delta_3\ge4\) gives \(\rho_4\ge12\).
	
	Finally, suppose that \(p_*=2\). If the first nonsimple restriction is
	nonexceptional, then \cref{lem:nonexceptional-profile} gives
	\(\rho_2\ge4\). Monotonicity through the first nonsimple step gives
	\(\delta_2\ge\delta_3\ge4\), and hence
	\(\rho_3=\rho_2+\delta_2\ge8\) and
	\(\rho_4=\rho_3+\delta_3\ge12\).
	
	If the first nonsimple restriction is exceptional, then the preceding
	restriction \(H_4\to H_3\) is simple and has \(\delta_3\ge4\).
	Its residual space \(Q_3\) satisfies \(L_3^*Q_3\subset H_3\).
	Choose a four-dimensional subspace \(Q'_3\subset Q_3\). By
	\cref{lem:filtered-macaulay},
	\[
	\rho_3\ge\dim(L_3^*Q'_3)\ge\gamma_3(4)=8,
	\]
	and therefore \(\rho_4=\rho_3+\delta_3\ge12\).
	
	Thus, in every remaining case, either the desired bound has already
	been proved, or
	\(\rho_3\ge8\) and \(\rho_4\ge12\). Since the borderline pair
	\((\rho_3,\rho_4)=(8,12)\) was excluded above, integrality gives
	\(\rho_3+\rho_4\ge21\). The orthogonal \(3+4\) decomposition then
	yields
	\[
	\rho\ge\rho_3+\rho_4\ge21.
	\]
	This completes the proof in dimension seven and hence in all remaining
	dimensions. Combining  \cref{thm: kappa big}, we finished the proof of  \cref{thm:main}.

\end{document}